\documentclass[12pt, USletter, reqno, twopage, makeidx]{amsart}
\usepackage[margin=3cm, marginpar=1cm]{geometry}

\usepackage{amsfonts, amstext, amsmath, amsthm, amscd, amssymb, enumitem, url}
\usepackage{mathabx}
\usepackage{tikz-cd}
\usetikzlibrary{shapes.geometric, arrows}
\usepackage{xcolor}
\usepackage{svg}
\usepackage{pdfpages}
\usepackage{pdflscape}
\usepackage{spverbatim}

\usepackage{geometry}
\usepackage{todonotes}
\usepackage{microtype}
\usepackage[parfill]{parskip}

\usepackage[colorlinks,citecolor=cyan,linkcolor=magenta]{hyperref}
\usepackage[nameinlink]{cleveref}

\usepackage{array}
\newcolumntype{C}[1]{>{\centering\let\newline\\\arraybackslash\hspace{0pt}}m{#1}}
\usepackage{multicol}
\usepackage{multirow}
\usepackage{makecell}
\usepackage{booktabs}

\newtheorem{thm}{Theorem}[section]
\newtheorem{cor}[thm]{Corollary}
\newtheorem{lemma}[thm]{Lemma}
\newtheorem{prop}[thm]{Proposition}

\newtheorem{conj}[thm]{Conjecture}

\newtheorem{defn}[thm]{Definition}
\newtheorem{rmk}[thm]{Remark}
\newtheorem{eg}[thm]{Example}

\numberwithin{equation}{section}

\newcommand{\ol}{\overline}

\newcommand{\wt}{\widetilde}

\newcommand{\fns}{\footnotesize}

\renewcommand{\epsilon}{\varepsilon}

\usepackage{stmaryrd}
\newcommand{\cut}{\!\bbslash\!}

\renewcommand{\top}{\mathrm{top}}
\renewcommand{\bot}{\mathrm{bot}}

\newcommand{\brloc}{\mathrm{brloc}}

\newcommand{\boa}{\boldsymbol{\alpha}}
\newcommand{\bob}{\boldsymbol{\beta}}
\newcommand{\x}{\bold{x}}
\newcommand{\y}{\bold{y}}
\newcommand{\z}{\bold{z}}

\newcommand{\s}{\mathfrak{s}}
\newcommand{\xt}{\bold{x}^\text{top}}
\newcommand{\xb}{\bold{x}^\text{bot}}
\newcommand{\sym}{\text{Sym}}
\newcommand{\Z}{\mathbb{Z}}
\newcommand{\R}{\mathbb{R}}
\newcommand{\T}{\mathbb{T}}

\newcommand{\gS}{\mathfrak{S}}
\newcommand{\hd}{ (\Sigma, \boa, \bob)}
\newcommand{\e}{ \widetilde{\epsilon}}
\newcommand{\ts}{ \widetilde{\mathfrak{s}}}

\DeclareMathOperator{\spinc}{\text{spin}^\text{c}}
\DeclareMathOperator{\Spinc}{\text{Spin}^\text{c}}

\begin{document}

\title[Heegaard Floer theory and pseudo-Anosov flows II]{Heegaard Floer theory and pseudo-Anosov flows II: differential  and Fried pants}

\author{Antonio Alfieri}
\address{department of mathematics\\
University of Georgia at Athens \\
1023 D. W. Brooks Drive, Athens, GA 30605}
\email{alfieriantonio90@gmail.com}

\author{Chi Cheuk Tsang}
\address{Département de mathématiques \\
Université du Québec à Montréal \\
201 President Kennedy Avenue \\
Montréal, QC, Canada H2X 3Y7}
\email{tsang.chi\_cheuk@uqam.ca}

\begin{abstract}
In earlier work, relying on work of Agol-Guéritaud and Landry-Minsky-Taylor, we showed that given a pseudo-Anosov flow $(Y,\phi)$ and a collection of closed orbits $\mathcal{C}$ satisfying the `no perfect fit' condition, one can construct a special Heegaard diagram for the link complement  $Y^\sharp= Y \setminus \nu (\mathcal{C})$  framed by the degeneracy curves. 
In this paper, we demonstrate how the special combinatorics of this diagram can be used to understand the differential of the associated Heegaard Floer chain complex.
More specifically, we introduce a refinement of the $\text{spin}^\text{c}$-grading obstructing two Heegaard states from being connected by an effective domain.
We describe explicitly the subcomplexes in the refined gradings that represent irreducible multi-orbits, in the sense that they contain states corresponding to multi-orbits which cannot be resolved along Fried pants. In particular we show that the homology of these subcomplexes are 1-dimensional.
When specialized to the case of suspension flows our arguments prove some results in the spirit of Ni, Ghiggini, and Spano: the next-to-top non-zero sutured Floer group counts the number of periodic points of least period.
\end{abstract}

\maketitle

\setcounter{tocdepth}{3}
\makeatletter
\def\l@subsection{\@tocline{2}{0pt}{2.5pc}{5pc}{}}
\def\l@subsubsection{\@tocline{2}{0pt}{5pc}{7.5pc}{}}
\makeatother

\thispagestyle{empty}
\section{Introduction}

A flow $\phi$ on a closed oriented 3-manifold $Y$ is called \emph{pseudo-Anosov} if there exists a pair of singular 2-dimensional foliations $(\Lambda^s, \Lambda^u)$ whose leaves intersect transversely in the flow lines. The two foliations are required to be such that the flow lines in each leaf of $\Lambda^s$ are forward asymptotic, and the flow lines in each leaf of $\Lambda^u$ are backward asymptotic. 
A pseudo-Anosov flow $\phi$ is said to have \emph{no perfect fits} relative to a collection of closed orbits $\mathcal{C}$ if $\phi$ has no anti-homotopic orbits in the complement of $\mathcal{C}$.

Given such a pair $(\phi,\mathcal{C})$, work of Agol-Guéritaud, Schleimer-Segerman \cite{SS19}, and Landry-Minsky-Taylor \cite{LMT23} state that the dynamical information in $(\phi,\mathcal{C})$ can be encoded by a combinatorial device $B$, called a \emph{veering branched surface}, on the blown-up manifold $Y^\sharp = Y \backslash \nu(\mathcal{C})$. 
In the prequel to this paper \cite{AT25a} we showed that from $B$, one can construct a Heegaard diagram $\hd$ for $\mathcal{C}$ as a link in $Y$ framed by the degeneracy curves. 
In other words, $\hd$ is a Heegaard diagram for $Y^\sharp$ with a suture along each closed boundary orbit of the blown-up flow $\phi^\sharp$.

The focus of \cite{AT25a} was on the generators of the associated chain complex $SFC(\phi,\mathcal{C})$, i.e. the Heegaard states of the specific Heegaard diagrams we built.
Using the special combinatorics of $B$, we showed that each Heegaard state $\x$ corresponds to an embedded multi-loop $\mu_\x$ of the augmented dual graph $G_+$.
Furthermore, using the theory of veering branched surfaces we show that $\mu_\x$ is homotopic to a closed multi-orbit $\gamma_\x$ of the blown-up flow $\phi^\sharp$.

The focus of this paper will be on the differential of the chain complex $SFC(\phi,\mathcal{C})$.
If we had to  summarize our work  into one slogan,  that would be: \emph{A differential can appear between two generators $\x$ and $\y$ only if their multi-loops $\mu_\x$ and $\mu_\y$ can be deformed into each other across sectors, and/or by resolving their multi-orbits $\gamma_\x$ and $\gamma_\y$ along Fried pants}. Here by deforming across sectors we mean strumming, as specified in \cite[Section 4.4]{AT25a}, while by Fried pants resolution we mean a form of cobordism on orbits introduced in \cite{Fri83}, and recently used by Zung to define a certain combinatorial invariant of pseudo-Anosov flows \cite{Zun25}.

\subsection{The $\ts$-grading}

Part of our work consists in developing a technique to tell apart  generators in the differential of $SFC(\phi,\mathcal{C})$. We briefly summarize our ideas. 

First of all in \cite{AT25a} we showed that the Heegaard diagram $(\Sigma, \boa, \bob)$ is a balanced, admissible sutured Heegaard diagram, and that it makes sense to consider the sutured Heegaard Floer chain group $SFC(\phi, \mathcal{C})=CF\hd$ as in \cite{Juh06}. Let $\mathfrak{S}(\Sigma, \boa, \bob)$ denote the generators of $SFC(\phi, \mathcal{C})$, that is the Heegaard states of the diagram $\hd$.

In \cite{OS04c} Ozsv\' ath and Szab\' o introduced a map 
\[\epsilon: \mathfrak{S}(\Sigma, \boa, \bob) \times \mathfrak{S}(\Sigma, \boa, \bob)  \to H_1(M)\]
that can be used to obstruct two generators $\x$ and $\y$ from being connected by some Heegaard Floer differential. More specifically, in \cite{OS04c} they prove the following fact. See also \cite[Definition 4.6]{Juh06}.
\begin{thm}[Ozsv\' ath-Szab\' o]
    Two generators $\x$ and $\y$ are connected by a $C^\infty$ Whitney disk if and only if $\epsilon(\x,\y)=0$. Consequently, if $\epsilon(\x,\y)\neq 0$ then  in the Heegaard Floer differential 
    \[\partial \x = \sum_{\y \in \mathfrak{S}\hd} c(\x, \y) \cdot  \y \ , \]
    the coefficient $c(\x,\y)$  vanishes.  
\end{thm}
Note that the Ozsv\'ath-Szab\'o $\epsilon$-map satisfies some functional equations:
\begin{equation} \label{eq:epsilonproperties}
\epsilon(\x,\x)=0\ , \ \ \  \epsilon(\x,\y)=\epsilon(\y,\x)\ , \ \ \  \epsilon(\x,\y)+\epsilon(\y,\z)= \epsilon(\x,\z)
\end{equation}
that allows one to define an equivalence relation: two generators $\x$ and $\y$ are declared to be $\epsilon$-equivalent $\x\sim \y$ if and only if $\epsilon(\x,\y)=0$. The equivalence classes of this relation partition generators according to their $\spinc$-grading. Indeed, the $\spinc$-grading can be thought of as a map
\[\Spinc\hd := \mathfrak{S}\hd/\sim \  \longrightarrow \Spinc(M, \Gamma)\]
labeling each equivalence class with a $\spinc$-structure of the sutured manifold $(M,\Gamma)$.

In the specific set-up of veering branched surfaces, we showed in \cite{AT25a} that there are two preferential Heegaard states which we call $\xt$ and $\xb$. Using the bottom generator $\xb$ as \emph{reference generator} we can turn the $\epsilon$-map into a map 
\[\epsilon: \mathfrak{S}(\Sigma, \boa, \bob) \to H_1(M)\] 
by specifying the second argument: $\epsilon(\x)=\epsilon(\x, \xb)$. It follows from \Cref{eq:epsilonproperties} that in this notation two generators $\x$ and $\y$ are $\epsilon$-equivalent if and only if $\epsilon(\x)=\epsilon(\y)$. Furthermore, since $M$ retracts onto $B$, there is an identification
\[H_1(M)= \frac{H_1(\Sigma)}{ \text{Span}_\Z \langle[\alpha_1], \dots, [\alpha_d],[\beta_1], \dots , [\beta_d]  \rangle } = \frac{H_1(G)}{ \text{Span}_\Z\langle [\partial S_1], \dots , [\partial S_d]  \rangle} \ ,\] 
where $G$ denotes the $1$-skeleton of $B$, and   $S_1, \dots, S_d$ denotes the sectors of $B$. 

In \Cref{sec:wtsgrading} we show that the $\epsilon$-map lifts to a multi-valued map
\begin{center}
   \begin{tikzcd}[column sep=small]
     & H_1(G)\arrow{d}{i_*}  \\
\mathfrak{S}(\Sigma, \boa, \bob) \arrow{ru}{\widetilde{\epsilon}}  
 \arrow{r}{\epsilon} & H_1(M)
\end{tikzcd} 
\end{center}
that can be used to obstruct the existence of effective domains.
 
\begin{rmk} 
Here by a multi-valued map $\widetilde{f}: X\to Y$ we mean a genuine map from $X$ to the finite subsets of $Y$. For a multi-valued map to have  the lifting  property
 \begin{center}
   \begin{tikzcd}[column sep=small]
     & Y\arrow{d}{\pi}  \\
X\arrow{ru}{\widetilde{f}}  
 \arrow{r}{f} & Z
\end{tikzcd} 
\end{center}
 we mean that $f(x)=\pi(y)$, for all $y \in \widetilde{f}(x)$.  
 \end{rmk}

One of our main technical results can be stated as follows.

\begin{lemma} \label{lemma:introepsilondifference}
Suppose that $\hd$ is the Heegaard diagram associated to a veering branched surface, and that $\x$ and $\y$ are two Heegaard states. If there exists an effective domain $D$ with $n_\z(D)=0$ connecting $\x$ to $\y$, then $\widetilde{\epsilon}(\x) \cap \widetilde{\epsilon}(\y) \neq \varnothing$. 
\end{lemma}

We use the $\e$-map to define a finer equivalence relation than Ozsv\' ath-Szab\' o  $\spinc$-grading encoded by $\epsilon$-equivalence. This allows us to partition  the generating set of the Heegaard Floer chain group  into smaller equivalence classes. 
If $\widetilde{\Spinc}\hd$ denotes the set of equivalence classes of the $\e$-equivalence relation then there are maps 
\[\mathfrak{S}\hd \to \widetilde{\Spinc}\hd \to \Spinc\hd \ .\]
The first map is what we call the $\ts$-grading. As a consequence of \Cref{lemma:introepsilondifference}, we have a direct-sum decomposition 
$$SFC(\phi,\mathcal{C}) = \bigoplus_{\ts \in \widetilde{\Spinc}\hd} SFC(\phi,\mathcal{C},\ts)$$
of chain-complexes. To illustrate the $\ts$-grading we use the example of the figure-eight knot complement.

\begin{eg}[Figure-eight knot] 
Let $\phi$ be the suspension flow on the mapping torus of the torus homeomorphism corresponding to the matrix $\begin{bmatrix} 2 & 1 \\ 1 & 1 \end{bmatrix} \in SL_2(\Z)$. This has a fixed point at the origin, which suspends to a closed orbit $\gamma$. One can see that $\phi$ has no perfect fits with respect to $\mathcal{C}=\{\gamma\}$. The blow-up flow $\phi^\#$ is a flow on the complement of the figure-eight knot $K \subset S^3$. This has a branched surface $B$ whose dual graph $G$ looks as in \Cref{fig:figureeighteg} top left. 

Since $G$ has two branch loops, after cutting along the $\alpha$-curves, the Heegaard diagram decomposes into two punctured annuli with 3 strands of $\beta$-curves going along their cores. See \Cref{fig:figureeighteg} bottom left. Inspecting the diagram one can see that there is a total of 10 Heegaard states. Two of them are $\xt = u_1v_1$ and  $\xb = u_3v_3$, located in $\s$-gradings $\s_{\phi^\sharp}$ and $\overline{\s_{\phi^\sharp}}$, respectively. The other eight generators $\{u_2v_3,u_4v_3,u_3v_2,u_3v_4,u_2v_2,u_4v_2,u_2v_4,u_4v_4\}$ are partitioned into two other $\spinc$-structures. In each of these $\spinc$-gradings there is a total of four generators. But one of the $\spinc$-gradings further partitions in 4 distinct $\ts$-gradings, while the other does not further decompose. 
See \Cref{fig:figureeighteg} right.
\end{eg}

\begin{figure}
    \centering
    \selectfont\fontsize{6pt}{6pt}
    \resizebox{!}{6cm}{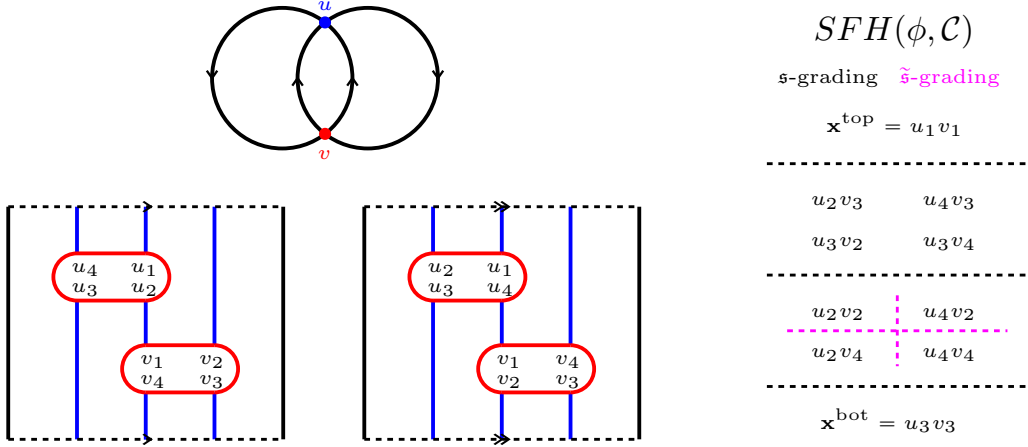}
    \caption{The figure-eight knot complement $Y^\sharp$ admits a veering branched surface $B$. Top left: Dual graph of $B$. Bottom left: Heegaard diagram of $Y^\sharp$ associated to $B$. Right: $SFH(Y^\sharp)$ computed using this Heegaard diagram, along with the $\mathfrak{s}$- and $\ts$-gradings.}
    \label{fig:figureeighteg}
\end{figure}

\subsection{Relation with Fried pants} 

Let $\phi$ be a pseudo-Anosov flow. An \emph{(immersed) Fried pant}  is an immersed pair of pants that is transverse to $\phi$ in its interior and whose boundary components are closed orbits of $\phi$.
These were introduced in \cite{Fri83} as a way to build transverse surfaces to pseudo-Anosov flows. For this paper, we will consider these as a form of cobordism between closed orbits.
We say that a multi-orbit is \emph{pants-irreducible} if it is not cobordant with another multi-orbit through a Fried pant.
In \Cref{sec:friedpants} we study generators corresponding to pants-irreducible multi-orbits.

We show that if $\gamma$ is pants-irreducible then all generators $\x \in \mathfrak{S}\hd$ with $\gamma_\x=\gamma$ make up a class $\ts\in \widetilde{\Spinc}\hd$. This gives us a direct summand 
\[SFC(\phi,\mathcal{C},\gamma)=SFC(\phi,\mathcal{C}, \ts)\] 
of the Heegaard Floer chain complex associated to the flow through the veering branched surface.

\begin{thm} \label{thm:introdim1}
Let $\gamma$ be a pants-irreducible multi-orbit of $\phi^\sharp$. Then 
\[H_*(SFC(\phi,\mathcal{C},\gamma)) \cong \mathbb{F} \ .\]
\end{thm}
As an immediate consequence of our work we have the following corollary of dynamical significance.
\begin{cor} \label{cor:intropantsirreduciblelowerbound}
Let $\phi$ be a pseudo-Anosov flow on a closed oriented 3-manifold $Y$ and let $\mathcal{C}$ be a collection of closed orbits such that $\phi$ has no perfect fits relative to $\mathcal{C}$. 
Then 
$$\dim SFH(Y^\sharp) \geq \text{\# pants-irreducible multi-orbits of $\phi^\sharp$}.$$
\end{cor}

\Cref{thm:introdim1} is actually a special case of a more general theorem involving a special class of multi-orbits we call sleek. 

\subsection{Sleek multi-orbits} \label{subsec:introsleek}
We define a multi-orbit of $\phi^\sharp$ to be \emph{sleek} if all of its corresponding multi-loops of $G$ are embedded.
This terminology is motivated from the imagery that such multi-loops never get `tangled-up' and become non-embedded even as one sweeps them across sectors.

\begin{figure}[t]
    \centering
    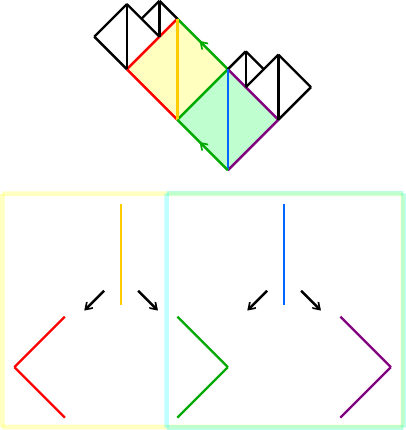
    \caption{A simple example illustrating the induction on cores in the proof of \Cref{thm:insulatedorbithomology}. In this example there are two cores: the one shaded in yellow contains the other shaded in blue (with the area of overlap appearing green). The portions of the chain complex corresponding to the yellow and blue cores are indicated by the yellow and blue boxes respectively. }
    \label{fig:introzigzag}
\end{figure}

For example, a pants-irreducible multi-orbit is sleek, since if a multi-orbit $\gamma$ corresponds to a non-embedded multi-loop $c$, one can resolve $c$ at a vertex, which corresponds to resolving $\gamma$ along a Fried pants.

Generalizing the discussion above, if $\gamma$ is sleek then all generators $\x \in \mathfrak{S}\hd$ with $\gamma_\x=\gamma$ generate a direct summand $SFC(\phi,\mathcal{C},\gamma)$.
The following theorem, proved in \Cref{sec:sleeksummands}, implies \Cref{thm:introdim1}.

\begin{thm}\label{thm:introinsulatedorbithomology}
Let $\gamma$ be a sleek multi-orbit of $\phi^\sharp$. Then 
\[H_*(SFC(\phi,\mathcal{C},\gamma)) \cong \mathbb{F} \ .\]
\end{thm}

Here is a quick outline of the proof of \Cref{thm:insulatedorbithomology}:
Consider the \emph{dynamic annulus} or \emph{Möbius band} $D(\gamma)$ that contains $\gamma$. 
We combinatorially define a chain complex $CC(C)$ for each compact core $C \subset D(\gamma)$ whose boundary lies along $G$.
The generators of $CC(C)$ are the multi-loops of $G_+$ corresponding to $\gamma$ that lie in $C$, while the differentials are given by strumming.

A large enough core $C$ contains all multi-loops that represent $\gamma$. 
For such a core $C$, we show that the combinatorial chain complex $CC(C)$ equals the summand $SFC(\phi,\mathcal{C},\widetilde{\s})$.
These chain complexes already have the same generators, so it suffices to identify the differentials in $SFC(\phi,\mathcal{C},\widetilde{\s})$. To do this, we make great use of \Cref{lemma:introepsilondifference}.

We then run an induction argument by showing that if $C'$ is a sub-core of $C$ with one less sector, then the inclusion $CC(C') \subset CC(C)$ induces an isomorphism on homology. 
The proof of this is essentially a classic zig-zag argument. See \Cref{fig:introzigzag} for a very simple example.
Hence we reduce to the case when $C$ just consists of one orbit, in which case it is clear that $H_*(CC(C)) = \mathbb{F}$.

\subsection{Applications to suspension flows} \Cref{thm:introdim1} has a neat application to suspension pseudo-Anosov flows. Recall that if $f$ is a pseudo-Anosov map on a closed oriented surface $S$, then the suspension flow $\phi_f$ on the mapping torus $Y_f$ is pseudo-Anosov. 
Furthermore, $\phi_f$ has no perfect fits relative to any collection of closed orbits $\mathcal{C}$ that contains the singular orbits.

With the choice of $\mathcal{C}$ fixed, the blown-up flow $\phi_f^\sharp$ is the suspension flow of the map $f^\sharp: S^\sharp \to S^\sharp$ obtained by blowing up $f$ at the points where $\mathcal{C}$ meets $S$. 
Let $e = \sum_{x \in S \cap \mathcal{C}} p_x$, where the sum is taken over the periodic points $x \in S \cap \mathcal{C}$, and we suppose each $x$ is $p_x$-pronged.

In \cite{AT25a}, we showed that under the identification $\Spinc(Y^\sharp) \cong H_1(Y^\sharp)$ that sends $\overline{\s_{\phi^\sharp}}$ to $0$, we have
$$\dim SFH(Y^\sharp,n) := \dim \bigoplus_{\langle \mathfrak{s}, [S^\sharp] \rangle = n} SFH(Y^\sharp,\mathfrak{s}) =
\begin{cases}
0 & \text{if $n > \frac{3e}{2}$} \\
1 & \text{if $n = \frac{3e}{2}$} \\
1 & \text{if $n = 0$} \\
0 & \text{if $n < 0$}
\end{cases}.$$
In particular, one should think of $SFH(Y^\sharp,\frac{3e}{2})$ as the top grading and $SFH(Y^\sharp,0)$ as the bottom grading.

Let $P$ be the minimum period among periodic points of $f^\sharp$.
For example, if $f^\sharp$ has fixed points, then $P=1$. If $f^\sharp$ does not have fixed points but has periodic points of period 2, then $P=2$, and so on.
Then by definition there are no closed orbits of $\phi^\sharp_f$ that intersect the fiber surface $S^\sharp$ in the mapping torus $Y^\sharp_f$ for less than $P$ times.
Furthermore, any closed orbit that intersects $S^\sharp$ between $P$ and $2P-1$ times must be pants-irreducible, since otherwise at least one of the resolved orbits will intersect $S^\sharp$ for less than $P$ times.
This allows us to apply \Cref{thm:introdim1} to deduce the following theorem.

\begin{thm} \label{thm:introperiodicpoint}
Let $f$ be a pseudo-Anosov map on a closed oriented surface $S$. Let $\mathcal{C}$ be a collection of closed orbits of the suspension flow $\phi_f$ containing the singular orbits. 

Let $e = \sum_{x \in S \cap \mathcal{C}} p_x$, where the sum is taken over the periodic points $x \in S \cap \mathcal{C}$, and we suppose each $x$ is $p_x$-pronged.
Let $P$ be the minimum period among periodic points of the blown-up map $f^\sharp$. 

Then under the identification $\Spinc(Y^\sharp) \cong H_1(Y^\sharp)$ that sends $\overline{\s_{\phi^\sharp}}$ to $0$, we have
\begin{align*}
\dim SFH(Y^\sharp,n) &:= \bigoplus_{\langle \s,[S^\sharp] \rangle = n} SFH(Y^\sharp,\s) \\
&= \text{\# periodic orbits of $\phi_f^\sharp$ that intersect $S^\sharp$ for $n$ times} \\
&= \frac{1}{n} \cdot \text{\# periodic points of $f^\sharp$ of period $n$}
\end{align*}
for $n = P,\dots,2P-1$ and $n=\frac{3e}{2}-(2P-1),\dots,\frac{3e}{2}-P$.
\end{thm}

Specifically, the next-to-top nonzero grading counts the number of periodic points of $f^\sharp$ of the least period. 
In particular, it is always true that the second-to-top grading counts the number of fixed points of $f^\sharp$.

In the special case when $\mathcal{C}$ is a fibered knot in a 3-manifold, this matches up with the results of Ni \cite{Ni22} and Ghiggini-Spano \cite{GS22}.

\begin{cor} \label{cor:introknotfixedpoint}
Let $K$ be a fibered hyperbolic knot in a closed oriented 3-manifold $Y$. Let $g$ be the genus of the fiber surface, and let $f^\sharp$ be the monodromy of the fibering. 
Suppose $f^\sharp$ does not have interior singularities, and suppose the degeneracy slope is the meridian of $K$.
Then 
$$\dim \widehat{HFK}(Y,K,g-1) = (\text{\# interior fixed points of $f^\sharp$})+4g-1.$$
\end{cor}

\subsection{Relation with Zung's chain complex} \label{subsec:introzung}

In \cite{Zun25}, Zung constructs a chain complex $CA(\phi)$ associated to an Anosov flow $\phi$, defined over a Novikov ring.
The generators of $CA(\phi)$ are the closed multi-orbits of $\phi$. 
In the case of no perfect fits, the differential of $CA(\phi)$ is defined by counting embedded Fried pants (weighted by the flux of the flow through the pants).

It is interesting to note the differences between our chain complex $SFC(\phi,\mathcal{C})$ we study in this paper, and the chain complex $CA(\phi)$ defined by Zung.
\begin{itemize}
    \item To define $SFC(\phi,\mathcal{C})$, one has to choose a collection of closed orbits $\mathcal{C}$ relative to which $\phi$ has no perfect fits. As long as $\phi$ is transitive, such a choice always exists. However, for general $\phi$, we do not know if a canonical choice exists. For the rest of this discussion, we simply fix some choice of $\mathcal{C}$.
    \item The generators of $SFC(\phi,\mathcal{C})$ have associated multi-orbits. However, different generators could have the same associated multi-orbit. Also, only finitely many multi-orbits appear. Meanwhile, the generators of $CA(\phi)$ are exactly the multi-orbits of $\phi$, for which there are infinitely many.
    \item The differential of $SFC(\phi,\mathcal{C})$ arises from strumming the multi-loops $\mu_\x$ across sectors and resolving the multi-orbits along Fried pants $\gamma_\x$. 
    \begin{itemize}
        \item Morally, the act of strumming the multi-loops removes the redundancy of having the same multi-orbit correspond to multiple generators. 
        \item It is possible that the differential picks up compositions of Fried pants, since two multi-loops representing the same homology class may be related by multiple resolutions. These correspond to surfaces with larger number of boundary components or genus that are transverse to $\phi^\sharp$ in their interiors. 
    \end{itemize}
    Meanwhile, the differential of $CA(\phi)$ only counts Fried pants.
    \begin{itemize}
        \item Morally, the `long' orbits, i.e. the ones that do not appear as generators in $SFC(\phi,\mathcal{C})$ cancel each other out. This is an instance of a general algebro-combinatorial principle that since non-embedded loops can be resolved in an even number of ways, they cancel each other out.
        \item Surfaces with larger number of boundary components or genus are disregarded for ECH index reasons.
    \end{itemize}
\end{itemize}

In private communication, Zung has offered the following conjecture to reconcile the two chain complexes.

\begin{conj}[Zung]
There is a chain complex $C\Omega(\phi)$ generated by all multi-loops of $G_+$ (not just the embedded ones) and whose differential counts strumming of multi-loops and cobordisms of orbits along transverse surfaces.

The differential of $C\Omega(\phi)$ is filtered in two ways:
\begin{enumerate}
    \item By grading the generators according to their non-embeddedness, then separating the differential into the part $d^0_{SFC}$ that goes between generators of the same grading, and other parts $d^n_{SFC}$ that go between generators with various other grading differences. 
    \item By separating the differential into the part $d^0_{CA}$ counting strumming, the part $d^1_{CA}$ counting Fried pants, and other parts $d^n_{CA}$ counting various other transverse surfaces.
\end{enumerate}

The spectral sequence associated to filtration (1) collapses at the $E_1$ page and gives $H_*(SFC(\phi,\mathcal{C}))$. Morally, this will be because non-embedded multi-loops can be resolved in an even number of ways, which cancel each other out in $d^0_{SFC}$.
The spectral sequence associated to filtration (2) recovers $CA(\phi)$ at its $E_1$ page. Morally, this will be because $d^0_{CA}$ cancels out all but one multi-loop for each multi-orbit, similarly as \Cref{thm:introinsulatedorbithomology}.
\end{conj}

This conjecture would explain why both $SFC(\phi,\mathcal{C})$ and $CA(\phi)$ categorify (slight variations of) the zeta function of $\phi$.

{\it {\bf Acknowledgments.} The first author would like to thank  Andras Stipsicz for some useful correspondence.
The second author would like to thank Ian Agol and Jonathan Zung for many enlightening conversations on Floer theory, pseudo-Anosov flows, and veering triangulations, many of which influenced the directions and results of this paper.
The ideas we expose in this paper were developed when the two authors were postdoctoral fellows at CIRGET, the geometry and topology lab of Université du Québec à Montréal. We would like to thank Steven Boyer for creating the fertile scientific environment in which this collaboration originated.}

{\bf Notational conventions.} Throughout this paper we follow the same notational conventions as in \cite{AT25a}.

\section{Background material and review of previous work}

In this section, we review some background material about veering branched surfaces, pseudo-Anosov flows, Heegaard Floer homology, and our previous paper \cite{AT25a}.

\subsection{Veering branched surfaces}

Let $M$ be a compact oriented 3-manifold with torus boundary components.
A \emph{branched surface} in $M$ is a 2-complex $B \subset M$ where every point in $B$ has a neighborhood smoothly modeled on a point in \Cref{fig:veeringcondition}.

The union of the nonmanifold points of $B$ is called the \emph{branch locus} of $B$ and is denoted by $\brloc(B)$. There is a natural structure of a 4-valent graph on $\brloc(B)$. The vertices of $\brloc(B)$ are called the \emph{triple points} of $B$. 
The \emph{sectors} of $B$ are the complementary regions of $\brloc(B)$ in $B$.
Each edge of $\brloc(B)$ has a canonical coorientation on $B$, which we call the \emph{maw coorientation}, given by the direction from the side with more sectors to the side with less sectors, as indicated in \Cref{fig:veeringcondition}.

A \emph{branch loop} of $B$ is a smoothly immersed closed curve in $\brloc(B)$. 
Each sector of $B$ carries a natural structure of surfaces with corners. The corners are where the boundary switches from lying on one branch loop to another branch loop locally.

\begin{defn} \label{defn:vbs}
A branched surface $B$ in $M$ is \emph{veering} if:
\begin{enumerate}
    \item Each sector of $B$ is homeomorphic to a disc.
    \item Each component of $M \cut B$ is a \emph{cusped torus shell}, i.e. a thickened torus with a nonempty collection of cusp curves on one of its boundary components.
    \item There is a choice of orientation on each branch loop so that at each triple point, the orientation of each branch loop induces the maw coorientation on the other branch loop. See \Cref{fig:veeringcondition}.
\end{enumerate}
\end{defn}

\begin{figure}
    \centering
    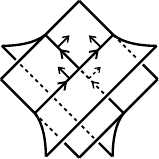
    \caption{A local picture around a triple point in a branched surface satisfying \Cref{defn:vbs}(3).}
    \label{fig:veeringcondition}
\end{figure}

The choice of orientations in (3) is unique if it exists. As such, given any veering branched surface $B$, we will implicitly adopt those orientations on its branch loops.

\subsection{Combinatorics of veering branched surfaces}

The edges of $\brloc(B)$ inherit orientations from its branch loops.
This upgrades the structure of $\brloc(B)$ as a 4-valent graph to a $(2,2)$-valent directed graph, i.e. a directed graph where every vertex has 2 incoming edges and 2 outgoing edges. We refer to $\brloc(B)$ with this upgraded structure as the \emph{dual graph} of $B$ and denote it by $G$.

\begin{prop} \label{prop:vbssector}
Each sector $S$ is a diamond, i.e. a disc with four corners, where the orientations on the sides all point from the bottom corner to the opposite top corner. See \Cref{fig:vbssector}. 
\end{prop}

\begin{figure}
    \centering
    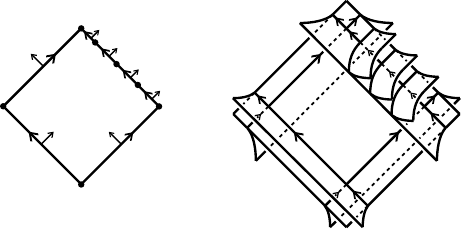
    \caption{Each sector of a veering branched surface is a diamond.}
    \label{fig:vbssector}
\end{figure}

We define the \emph{augmented dual graph} $G_+$ to be the dual graph $G$ with an additional directed edge going from the bottom corner to the top corner in each sector.

\begin{defn}[Strum] \label{defn:strum}
Let $c$ be a loop of $G_+$ containing an edge $e$ that does not lie in $G$, i.e. $e$ connects the bottom vertex to the top vertex in some sector $S$. We can define another loop $\widetilde{c}$ of $G_+$ by replacing $e$ in $c$ by the path in $G$ that traverses a bottom side and a top side of $S$.
We refer to $\widetilde{c}$ as a \emph{strum} of $c$ across $S$. 
See \Cref{fig:strum}.
\end{defn}

\begin{figure}
    \centering
    \selectfont\fontsize{8pt}{8pt}
\begingroup%
  \makeatletter%
  \providecommand\color[2][]{%
    \errmessage{(Inkscape) Color is used for the text in Inkscape, but the package 'color.sty' is not loaded}%
    \renewcommand\color[2][]{}%
  }%
  \providecommand\transparent[1]{%
    \errmessage{(Inkscape) Transparency is used (non-zero) for the text in Inkscape, but the package 'transparent.sty' is not loaded}%
    \renewcommand\transparent[1]{}%
  }%
  \providecommand\rotatebox[2]{#2}%
  \newcommand*\fsize{\dimexpr\f@size pt\relax}%
  \newcommand*\lineheight[1]{\fontsize{\fsize}{#1\fsize}\selectfont}%
  \ifx\svgwidth\undefined%
    \setlength{\unitlength}{147.67640602bp}%
    \ifx\svgscale\undefined%
      \relax%
    \else%
      \setlength{\unitlength}{\unitlength * \real{\svgscale}}%
    \fi%
  \else%
    \setlength{\unitlength}{\svgwidth}%
  \fi%
  \global\let\svgwidth\undefined%
  \global\let\svgscale\undefined%
  \makeatother%
  \begin{picture}(1,0.65114128)%
    \lineheight{1}%
    \setlength\tabcolsep{0pt}%
    \put(0,0){\includegraphics[width=\unitlength,page=1]{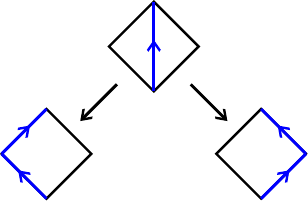}}%
    \put(0.17200929,0.34456134){\color[rgb]{0,0,0}\makebox(0,0)[lt]{\lineheight{1.25}\smash{\begin{tabular}[t]{l}strum\end{tabular}}}}%
    \put(0.67549017,0.34456134){\color[rgb]{0,0,0}\makebox(0,0)[lt]{\lineheight{1.25}\smash{\begin{tabular}[t]{l}strum\end{tabular}}}}%
  \end{picture}%
\endgroup%

    \caption{Each edge of $G_+ \backslash G$ can be strum into a path of $G$ in two ways.}
    \label{fig:strum}
\end{figure}

\begin{defn}[Sweep-equivalent] \label{defn:sweepequivalent}
We say that two loops in $G_+$ are \emph{sweep-equivalent} if they are related by a finite sequence of strums. 
More precisely, if we have a sequence of loops $c_0,\dots,c_n$ in $G_+$ where for each $i$, either $c_{i-1}$ is a strum of $c_i$ or $c_i$ is a strum of $c_{i-1}$, then we say that $c_0$ and $c_n$ are sweep-equivalent.
\end{defn}

A triple point $v$ of $B$ is \emph{blue} or \emph{red} depending on whether $B$ is locally of the form in \Cref{fig:vbscolor} left or right at $v$. 
Here the orientation on $M$ is used to distinguish between the two local pictures.

\begin{defn}[Blue/red, toggle, fan sectors] \label{defn:blueredsectors}
A sector is \emph{blue/red} if its top corner is blue/red, respectively.
We say that a sector is \emph{toggle} if the colors of its top and bottom corners are different. We say that it is \emph{fan} if the colors of its top and bottom corners are the same.
\end{defn}

\begin{figure}
    \centering
\begingroup%
  \makeatletter%
  \providecommand\color[2][]{%
    \errmessage{(Inkscape) Color is used for the text in Inkscape, but the package 'color.sty' is not loaded}%
    \renewcommand\color[2][]{}%
  }%
  \providecommand\transparent[1]{%
    \errmessage{(Inkscape) Transparency is used (non-zero) for the text in Inkscape, but the package 'transparent.sty' is not loaded}%
    \renewcommand\transparent[1]{}%
  }%
  \providecommand\rotatebox[2]{#2}%
  \newcommand*\fsize{\dimexpr\f@size pt\relax}%
  \newcommand*\lineheight[1]{\fontsize{\fsize}{#1\fsize}\selectfont}%
  \ifx\svgwidth\undefined%
    \setlength{\unitlength}{175.35233986bp}%
    \ifx\svgscale\undefined%
      \relax%
    \else%
      \setlength{\unitlength}{\unitlength * \real{\svgscale}}%
    \fi%
  \else%
    \setlength{\unitlength}{\svgwidth}%
  \fi%
  \global\let\svgwidth\undefined%
  \global\let\svgscale\undefined%
  \makeatother%
  \begin{picture}(1,0.5367177)%
    \lineheight{1}%
    \setlength\tabcolsep{0pt}%
    \put(0,0){\includegraphics[width=\unitlength,page=1]{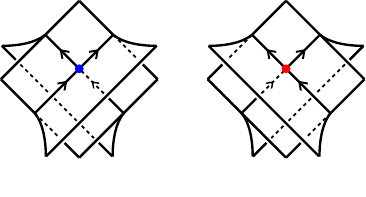}}%
    \put(0.16083516,0.00086324){\color[rgb]{0,0,1}\makebox(0,0)[lt]{\lineheight{1.25}\smash{\begin{tabular}[t]{l}bLue\end{tabular}}}}%
    \put(0.76408306,0.00086324){\color[rgb]{1,0,0}\makebox(0,0)[lt]{\lineheight{1.25}\smash{\begin{tabular}[t]{l}Red\end{tabular}}}}%
  \end{picture}%
\endgroup%

    \caption{The color of a triple point is determined by the local form of the branched surface. Here the orientation on $M$ is used to distinguish between the two local pictures.}
    \label{fig:vbscolor}
\end{figure}

\begin{prop} \label{prop:vbstogglefan}
Let $S$ be a blue/red sector of a veering branched surface $B$.
The top and side corners of $S$ are blue/red, respectively.
Each bottom side of $s$ is an edge of $\brloc(B)$.
Each top side of $S$ is the union of $\delta \geq 1$ edges of $\brloc(B)$.

Suppose a top side of $S$ is divided into edges $e_1,\dots,e_\delta$, listed from bottom to top. Let $S_i$ be the sector that has $e_i$ as a bottom side. If $\delta=1$, then $S_1$ is blue/red fan, respectively. If $\delta \geq 2$, then $S_1$ and $S_\delta$ are toggle while $S_i$ for $i=2,\dots,\delta-1$ are red/blue fan, respectively.
\end{prop}

See \Cref{fig:vbssectors} for an example. 

\begin{figure}
    \centering
    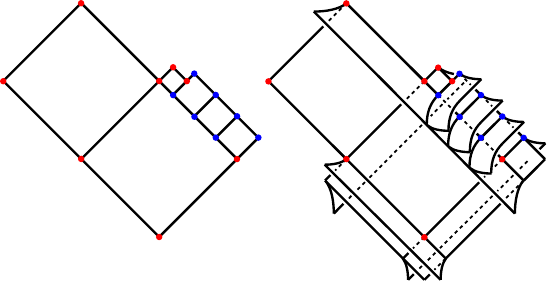
    \caption{An example of the local configuration of sectors in a veering branched surface.}
    \label{fig:vbssectors}
\end{figure}

\subsection{Pseudo-Anosov flows}

For the purposes of this paper, a \emph{pseudo-Anosov flow} on a closed 3-manifold $Y$ is a continuous flow $\phi$ for which there is a pair of singular 2-dimensional foliations $(\Lambda^s, \Lambda^u)$ whose leaves intersect transversely in the flow lines, such that the flow lines in each leaf of the \emph{stable foliation} $\Lambda^s$ are forward asymptotic, and the flow lines in each leaf of the \emph{unstable foliation} $\Lambda^u$ are backward asymptotic.
The singularity locus of the stable and unstable foliations coincide and is equal to the collection of \emph{singular} orbits of $\phi$.
See \Cref{fig:paflow} for local pictures of pseudo-Anosov flows.
The precise definition of a pseudo-Anosov flow involves some details on regularity; we refer the reader to \cite[Section 4]{FM01}.

\begin{figure}
    \centering
    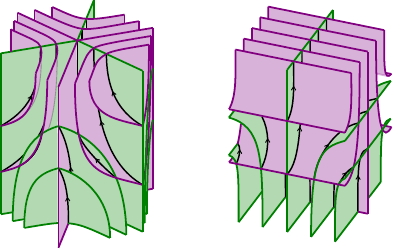
    \caption{Local picture of a pseudo-Anosov flow. Left: Near a singular orbit. Right: Away from a singular orbit.}
    \label{fig:paflow}
\end{figure}

Two pseudo-Anosov flows $\phi_1$ and $\phi_2$ on closed $3$-manifolds $Y_1$ and $Y_2$ are \emph{orbit equivalent} if there exists a homeomorphism $h:Y_1 \to Y_2$ that sends the orbits of $\phi_1$ to those of $\phi_2$ in an orientation-preserving way, but not necessarily preserving their parametrizations. 
We will consider orbit equivalent flows as being equivalent.

Given a finite collection $\mathcal{C}$ of closed orbits, one can define the \emph{blown-up flow} $\phi^\sharp$ on $Y^\sharp = Y \backslash \nu(\mathcal{C})$.
The restriction of $\phi^\sharp$ to the interior of $Y^\sharp$ can be identified with the restriction of $\phi$ to $Y \backslash \mathcal{C}$. 
Meanwhile, each boundary component $T$ of $M$ corresponds to a closed orbit $\gamma \in \mathcal{C}$. 
The stable and unstable foliations of $\gamma$ intersect $T$ in an equal number of circles. 
When restricted to $T$, $\phi^\sharp$ has repelling and attracting closed orbit exactly along these circles, respectively.
The \emph{degeneracy curves} of $\phi$ on $T$ is the isotopy class of the collection of repelling closed orbits on $T$, or equivalently, the isotopy class of the collection of attracting closed orbits on $T$.
We illustrate an example in \Cref{fig:paflowblowup} where we blow up along \Cref{fig:paflow} left. 

\begin{figure}
    \centering
    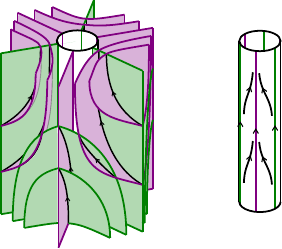
    \caption{Blowing up along \Cref{fig:paflow} left.}
    \label{fig:paflowblowup}
\end{figure}

We say that $\phi$ has \emph{no perfect fits} relative to $\mathcal{C}$ if the blown-up flow $\phi^\sharp$ does not admit two closed orbits $\gamma_1$ and $\gamma_2$ such that $\gamma_1$ is homotopic to $\gamma_2^{-1}$ in $Y^\sharp$.

\subsection{Correspondence between pseudo-Anosov flows and veering branched surfaces}

There is a correspondence theorem between pseudo-Anosov flows and veering branched surfaces.
We summarize the aspects of this theory that we will need in this paper below.
See \cite[Theorem 2.5]{AT25a} and \cite[Chapter 2]{Tsathesis} for more details.

\begin{thm} \label{thm:pavbscorr}
Let $\phi$ be a pseudo-Anosov flow on a closed oriented 3-manifold $Y$. Let $\mathcal{C}$ be a nonempty finite collection of closed orbits, containing the collection of singular orbits, such that $\phi$ has no perfect fits relative to $\mathcal{C}$.
Then $M = Y \backslash \nu(\mathcal{C})$ admits a veering branched surface $B(\phi,\mathcal{C})$ that carries the blown-up unstable foliation of $\phi$ in $M$.

Furthermore, there is a bijection $\mathcal{F}$ from the sweep-equivalence classes of loops in $G_+$ to the closed orbits of $\phi^\sharp$, such that $\mathcal{F}(c)$ is homotopic to $c$ in $M$ for each loop $c$.
\end{thm}

\subsection{Sutured manifolds} 

A \emph{sutured manifold} is a pair $(M, \Gamma)$ where $M$ is a 3-manifold whose boundary is decomposed into the union of two surfaces $R^+$, and $R^-$ meeting along a multicurve $\Gamma \subset \partial M$. We further require the components of $\Gamma$ to be homotopically non-trivial, and we call them \emph{sutures}. A sutured manifold $(M, \gamma)$ is \emph{balanced} if $\chi(R_+)=\chi(R_-)$ and each connected component of $\partial M$ contains at least one suture.

When $M$ is a 3-manifold with torus boundary components, a sutured structure consists simply of the choice of an even number of parallel curves on each boundary component of $M$. Note that these structures are automatically balanced since $R_+$ and $R_-$ in this case are collections of annuli, so $\chi(R_+)=\chi(R_-)=0$. 

If $(Y, \phi)$ is a pseudo-Anosov flow and $\mathcal{C}$ is a collection of closed orbits one can consider the blown-up flow $(Y^\sharp=Y \setminus \nu(C), \phi^\sharp)$. This has a natural sutured structure $\Gamma$ having one suture for each closed orbit of $\phi^\sharp$ on $\partial Y^\sharp$. 
Note that on the boundary component $T_\gamma$ corresponding to a closed orbit $\gamma \in \mathcal{C}$, the sutures have the same slope as the degeneracy slope of $\phi$ at $\gamma$.

\subsection{Sutured Heegaard diagrams}

Sutured manifolds can be described by \emph{sutured Heegaard diagrams}. 
These are triples $(\Sigma, \boa, \bob)$ consisting of (1) a compact, connected, oriented surface with boundary $\Sigma$, and (2) two collections of oriented, pairwise disjoint, simple closed curves $\boa=\{\alpha_1, \dots ,\alpha_n\}$ and  $\bob=\{\beta_1, \dots ,\beta_m\}$ contained in the interior of $\Sigma$.

Starting from a sutured Heegaard diagram $(\Sigma, \boa, \bob)$ with $n$ $\alpha$-curves and $m$ $\beta$-curves one can construct a sutured manifold $(M, \Gamma)$ by thickening $\Sigma$ to $\Sigma \times [-1,1]$ then attaching three-dimensional $2$-handles along $\alpha_i \times \{-1\}$ and $\beta_i \times \{1\}$, and finally taking $\Gamma = \partial \Sigma \times \{0\}$.

For the sutured manifold  $(M, \Gamma)$ associated to an Heegaard diagram $(\Sigma, \boa, \bob)$ to be balanced one must have that:
\begin{itemize}
    \item $|\boa|=|\bob|$, that is, there are as many $\alpha$-curves as many $\beta$-curves 
    \item every connected component of $\Sigma\ \setminus \bigcup_{i=1}^d \alpha_i$ contains at least a component of $\partial \Sigma$
     \item every connected component of $\Sigma\ \setminus \bigcup_{i=1}^d \beta_i$ contains at least a component of $\partial \Sigma$.
\end{itemize}
Consequently a diagram satisfying these conditions is called \emph{balanced}. It was shown in \cite[Proposition 2.13]{Juh06} that any balanced sutured manifold can be described by a sutured Heegaard diagram.

\subsection{Canonical diagrams of veering branched surfaces} \label{subsec:vbstoheegaarddiagram}

In the first part of this series \cite{AT25a} we showed that a veering branched surface $B$ in a 3-manifold $M$ gives rise to a canonical sutured Heegaard diagram $(\Sigma, \boa, \bob)$ for the sutured structure $\Gamma$ having two parallel sutures for each cusp curve of $B$.
We review the construction below.

We shall denote with $v_1, \dots , v_n$ the triple points of $B$, and by $C_1, \dots, C_k$, $C_i\simeq T^2 \times I$, the closure of the connected components of $M\setminus B$. Furthermore, let $G$ denote the $1$-skeleton of $B$, and $S_1, \dots, S_m$ its sectors. 

Let $U=N(G)$ be a tubular neighbourhood of the $1$-skeleton of $B$. 
Let $\Sigma_0=\partial U$ be the boundary of $U$. We place a set of curves $\boa = \{\alpha_1,\dots,\alpha_n\}$ in correspondence of the triple points of $B$ as suggested in \Cref{fig:vbstoheegaardtriplepoint}, and place a set of curves $\bob = \{\beta_1,\dots,\beta_n\}$ by setting $\beta_i = S_i \cap \Sigma_0$.

\begin{figure}
    \centering
    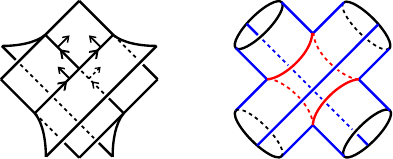
    \caption{The $\alpha$-curves are placed in correspondence with the triple points of $B$.}
    \label{fig:vbstoheegaardtriplepoint}
\end{figure}

After compressing the $\alpha$-curves along disks $D_i \subset U$, and the $\beta$-curves along the sectors we get a three-manifold $M_0\subset M$ with complement $M \setminus M_0$ a collection of solid tori $\{T_1, \dots, T_\ell\}$, one for each cusp curve on $\partial M$. 
Note that the core $\sigma_i$  of each solid torus $T_i\simeq \sigma_i\times D^2$ co-bounds an annulus $A_i \subset M$ with some cusp curve on $\partial M$. See \Cref{fig:vbstoheegaardcutannulus}. 

Let $M'_0\subset M_0 \subset M$ be the three-manifold obtained from $M_0$ by cutting along this collection of annuli. In the  manifold $M_0'$ the surface $\Sigma_0$ descends to an embedded surface $\Sigma \subset M_0'$ with boundary $\partial \Sigma \subset \partial M_0'$ . Furthermore, since the annuli $A_i$ intersect $\Sigma_0$ away from the $\alpha$- and the $\beta$-curves we can consider the  diagram $(\Sigma, \boa, \bob)$.

\begin{figure}
    \centering
\begingroup%
  \makeatletter%
  \providecommand\color[2][]{%
    \errmessage{(Inkscape) Color is used for the text in Inkscape, but the package 'color.sty' is not loaded}%
    \renewcommand\color[2][]{}%
  }%
  \providecommand\transparent[1]{%
    \errmessage{(Inkscape) Transparency is used (non-zero) for the text in Inkscape, but the package 'transparent.sty' is not loaded}%
    \renewcommand\transparent[1]{}%
  }%
  \providecommand\rotatebox[2]{#2}%
  \newcommand*\fsize{\dimexpr\f@size pt\relax}%
  \newcommand*\lineheight[1]{\fontsize{\fsize}{#1\fsize}\selectfont}%
  \ifx\svgwidth\undefined%
    \setlength{\unitlength}{366.05601874bp}%
    \ifx\svgscale\undefined%
      \relax%
    \else%
      \setlength{\unitlength}{\unitlength * \real{\svgscale}}%
    \fi%
  \else%
    \setlength{\unitlength}{\svgwidth}%
  \fi%
  \global\let\svgwidth\undefined%
  \global\let\svgscale\undefined%
  \makeatother%
  \begin{picture}(1,0.37523267)%
    \lineheight{1}%
    \setlength\tabcolsep{0pt}%
    \put(0,0){\includegraphics[width=\unitlength,page=1]{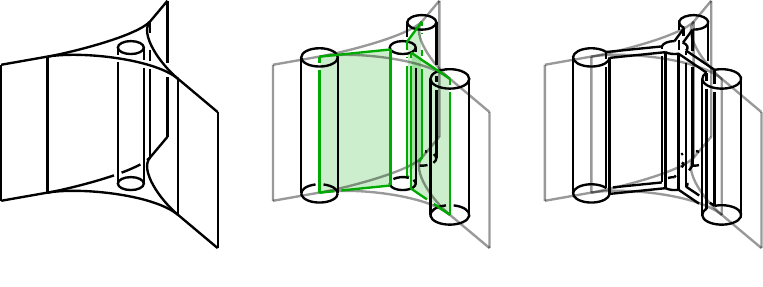}}%
    \put(0.13854616,0.00577843){\color[rgb]{0,0,0}\makebox(0,0)[lt]{\lineheight{1.25}\smash{\begin{tabular}[t]{l}$M$\end{tabular}}}}%
    \put(0.4950502,0.00577849){\color[rgb]{0,0,0}\makebox(0,0)[lt]{\lineheight{1.25}\smash{\begin{tabular}[t]{l}$M_0$\end{tabular}}}}%
    \put(0.85155383,0.00577849){\color[rgb]{0,0,0}\makebox(0,0)[lt]{\lineheight{1.25}\smash{\begin{tabular}[t]{l}$M'_0$\end{tabular}}}}%
  \end{picture}%
\endgroup%

    \caption{The 3-manifold obtained from $(\Sigma_0, \boa, \bob)$ is the submanifold $M_0 \subset M$ with complement $M \setminus M_0$ a collection of solid tori $\{T_1, \dots, T_\ell\}$, one for each cusp curve on $\partial M$. In $M_0$, we have a collection of annuli $A_i$, each with one boundary component along the core of $T_i$ and the other boundary component along $\partial M$.}
   \label{fig:vbstoheegaardcutannulus}
\end{figure}
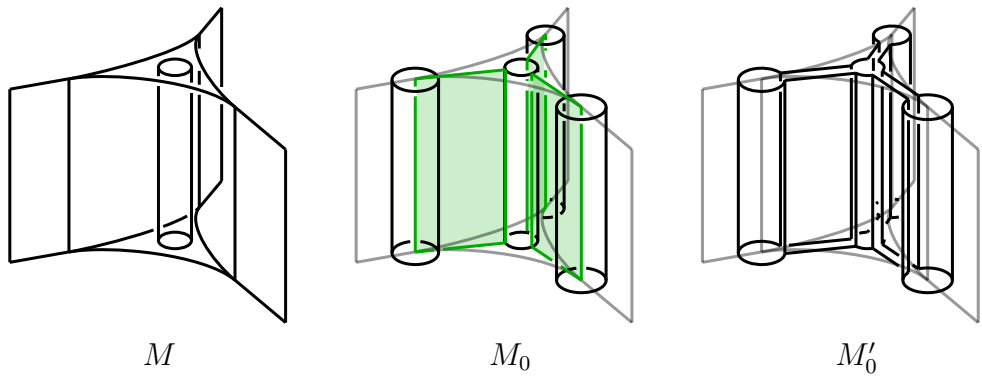

The combinatorics of the Heegaard diagram $(\Sigma, \boa, \bob)$ can be read from the dual graph of the veering branched surface.

\begin{prop}[Alfieri-Tsang {\cite[Proposition 5.8]{AT25a}}] \label{prop:globalcombinatorics}
Suppose that $B$ is a veering branched surface, and that $(\Sigma, \boa, \bob)$ is its corresponding Heegaard diagram. Then $\Sigma \setminus \boa$ decomposes as a union of punctured annuli $A_1, \dots , A_l$ in one-to-one correspondence with the branch loops of $B$. Furthermore, each annulus $A_i$ can be decomposed into a union of punctured rectangles, one for each $\alpha$-curve it contains in its closure.  
Each of these rectangles looks like one of the two local models depicted in \Cref{fig:heegaardcombinalphaglue}, and the specific local model is decided from the color of the triple point corresponding to the $\alpha$-curve. \qed
\end{prop}

\begin{figure}
    \centering
    \selectfont\fontsize{12pt}{12pt}
    \resizebox{!}{10.5cm}{
\begingroup%
  \makeatletter%
  \providecommand\color[2][]{%
    \errmessage{(Inkscape) Color is used for the text in Inkscape, but the package 'color.sty' is not loaded}%
    \renewcommand\color[2][]{}%
  }%
  \providecommand\transparent[1]{%
    \errmessage{(Inkscape) Transparency is used (non-zero) for the text in Inkscape, but the package 'transparent.sty' is not loaded}%
    \renewcommand\transparent[1]{}%
  }%
  \providecommand\rotatebox[2]{#2}%
  \newcommand*\fsize{\dimexpr\f@size pt\relax}%
  \newcommand*\lineheight[1]{\fontsize{\fsize}{#1\fsize}\selectfont}%
  \ifx\svgwidth\undefined%
    \setlength{\unitlength}{291.24206543bp}%
    \ifx\svgscale\undefined%
      \relax%
    \else%
      \setlength{\unitlength}{\unitlength * \real{\svgscale}}%
    \fi%
  \else%
    \setlength{\unitlength}{\svgwidth}%
  \fi%
  \global\let\svgwidth\undefined%
  \global\let\svgscale\undefined%
  \makeatother%
  \begin{picture}(1,0.89589967)%
    \lineheight{1}%
    \setlength\tabcolsep{0pt}%
    \put(0,0){\includegraphics[width=\unitlength,page=1]{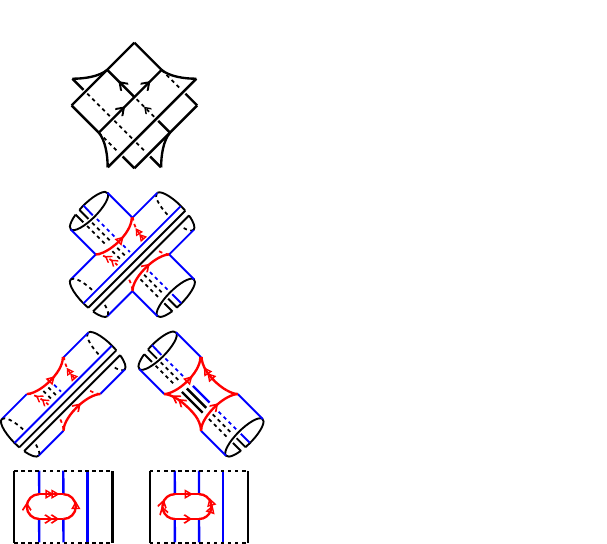}}%
    \put(0.18276867,0.8645952){\color[rgb]{0,0,0}\makebox(0,0)[lt]{\lineheight{1.25}\smash{\begin{tabular}[t]{l}bLue\end{tabular}}}}%
    \put(0,0){\includegraphics[width=\unitlength,page=2]{heegaardcombinalphaglue.pdf}}%
    \put(0.76225466,0.8645946){\color[rgb]{0,0,0}\makebox(0,0)[lt]{\lineheight{1.25}\smash{\begin{tabular}[t]{l}Red\end{tabular}}}}%
    \put(0,0){\includegraphics[width=\unitlength,page=3]{heegaardcombinalphaglue.pdf}}%
    \put(0.45711594,0.52857955){\color[rgb]{1,0.8,0}\makebox(0,0)[lt]{\lineheight{1.25}\smash{\begin{tabular}[t]{l}$\x^\bot$\end{tabular}}}}%
    \put(0.45769232,0.40556463){\color[rgb]{0.50196078,0,0.50196078}\makebox(0,0)[lt]{\lineheight{1.25}\smash{\begin{tabular}[t]{l}$\x^\top$\end{tabular}}}}%
  \end{picture}%
\endgroup%
}
    \caption{The local combinatorics of the Heegaard diagram near a triple point $v$ depends on the color of $v$.}
    \label{fig:heegaardcombinalphaglue}
\end{figure}

\subsection{Sutured Heegaard Floer homology}
In \cite{Juh06} Juhász introduced a variation of Heegaard Floer homology \cite{OS04c} called sutured Floer homology. This is a topological invariant of sutured manifolds in the form of a finite dimensional vector space. We briefly review the construction and highlight some important features of this invariant, for more details we refer to \cite{OSS}.

Let $(M, \Gamma)$ be a balanced sutured manifold. Choose a balanced Heegaard diagram $(\Sigma, \boa, \bob)$ for $(M, \Gamma)$, and pinch the boundary components of $\Sigma$ to turn it into a closed surface $\overline{\Sigma}$ with marked points $\z=\{z_1, \dots , z_k\}$.  

Let $d$ be the common number of $\alpha$- and $\beta$-curves in the diagram.
Let $\text{Sym}^d(\Sigma)$ be the $d$-fold symmetric product of $\overline{\Sigma}$. This is a smooth complex variety having as points all possible unordered $d$-tuples $\{x_1, \dots, x_d\}$ with $x_i \in \overline{\Sigma}$. We can consider the tori $\T_{\boa}=\alpha_1 \times \dots \times \alpha_d$, and $\T_{\bob}=\beta_1 \times \dots \times \beta_d$ as submanifolds of $\text{Sym}^d(\Sigma)$. 

The Floer chain complex $CF(\Sigma, \boa , \bob):= C_*(\T_{\boa}, \T_{\bob})$ is defined as the vector space formally generated by the intersection points of $\T_{\boa}$ and $\T_{\bob}$:
\[C_*(\T_{\boa}, \T_{\bob})= \bigoplus_{\bold{x}\in \T_{\boa}\cap \T_{\bob}} \mathbb{F} \cdot  \bold{x} \ , \]
equipped with differential
\begin{equation}\label{differential}
    \partial \x = \sum_{\y \in \T_{\boa} \cap \T_{\bob}} c(\x, \y ) \cdot \y \ ,
\end{equation}
where the coefficients $c(\x, \y)$ are defined as follows. First, one chooses  a suitable generic perturbation of the complex structure of $\text{Sym}^d(\Sigma)$, then for each pair of intersection points $\x,\y \in \T_{\boa}  \cap \T_{\bob}$, one considers all \emph{holomorphic disks} connecting $\x$ to $\y$, that is, all holomorphic maps 
\[u: D^2 \to \text{Sym}^d(\Sigma)\setminus  \ \bigcup_{i=1}^k\  \{z_i\} \times \sym^{d-1}(\overline{\Sigma})\] 
such that $u(-i)=\x$, $u(i)=\y$, $u(\{ x+iy : y>0  \}) \subset \T_{\boa}$, and  $u(\{ x+iy : y<0  \}) \subset \T_{\bob}$. Up to reparametrization of the source (Möbius transformations preserving $i$ and $-i$) there are only finitely many such maps with \emph{Maslov index one}. The coefficient $c(\x,\y)$ is defined as the total number of these holomorphic disks. 

\begin{thm}[Ozsváth-Szabó, Juhász] \label{thm:sfhwelldefined}
Suppose $(\Sigma, \boa, \bob)$ is an admissible, balanced Heegaard diagram for a sutured manifold $(M, \Gamma)$. Then the sutured Floer chain complex $CF(\Sigma, \boa, \bob)$ is well-defined and its homology only depends on the homeomorphism class of $(M,\Gamma)$. 
\end{thm}

Note that in \Cref{thm:sfhwelldefined} the sutured Heegaard diagram is required to be admissible. We shall not address this technicality here since it was already discussed in the first paper of this series.

\begin{prop}[Alfieri-Tsang {\cite[Proposition 5.14]{AT25a}}] The sutured Heegaard diagram of a veering branched surface is admissible. Consequently, its sutured Floer chain complex is well-defined.
\end{prop}

Suppose $B$ is a veering branched surface corresponding to a pseudo-Anosov flow $\phi$ and a collection of closed orbits $\mathcal{C}$. For the rest of this paper, we will write $SFC(\phi,\mathcal{C})$ for the sutured Floer chain complex associated to the sutured Heegaard diagram of $B$.

\subsection{Heegaard states} \label{subsec:heegaardstates}

The generators of the Heegaard Floer chain complex are called \emph{Heegaard states}. These are unordered $d$-tuples $\x=\{x_1, \dots, x_d\}$ such that $x_i\in \alpha_i \cap \beta_{\sigma(i)}$ for $i=1, \dots, d$, where $\sigma\in \mathfrak{S}_d$ is  some permutation of $\{1, \dots, d\}$. We shall denote by $\mathfrak{S}(\Sigma, \boa , \bob)$ the set of all Heegaard states of a diagram $(\Sigma, \boa , \bob)$.

\begin{prop}[Alfieri-Tsang {\cite[Proposition 4.6]{AT25a}}] \label{prop:generators}
Let $B$ a veering branched surface with sectors $S_1, \dots, S_n$ and triple points $v_1, \dots , v_n$. The Heegaard states $\mathfrak{S}(\Sigma, \boa , \bob)$ of the Heegaard diagram $(\Sigma, \boa , \bob)$ associated to $B$ are in natural correspondence with ways of assigning to each sector one of its four corners, so that each triple point is picked exactly once.
\end{prop}

The Floer chain complex of a veering branched surface $B$ has two preferential generators $\xt$ and $\xb$. The top generator $\xt$ of $CF(\Sigma, \boa, \bob)$ is defined by pairing each sector with its top corner, while the bottom generator $\xb$ is defined by pairing each sector with its bottom corner.

We can associate to each Heegaard state $\x \in \mathfrak{S}(\Sigma, \boa, \bob)$ a loop $\mu_\x$ of $G_+$ as follows.
Let $\x$ be a Heegaard state. We consider $\x$ to be a way of assigning each sector to one of its corner as in \Cref{prop:generators}.
For each sector $S_i$, let $\x(S_i)$ be the triple point assigned to $S_i$.
\begin{itemize}
    \item If $\x(S_i)$ is the bottom corner of $S_i$, take $e_{\x,i}$ to be the empty set.
    \item If $\x(S_i)$ is not the bottom corner of $S_i$, take $e_{\x,i}$ to be the edge of $G_+$ connecting the bottom corner of $S_i$ to $\x(S_i)$. (See \Cref{fig:statemultiloopdefn}.)
\end{itemize}
The union of edges $\mu_\x = \bigcup_i e_{\x,i}$ is an embedded multi-loop of $G_+$.

\begin{figure}
    \centering
    \selectfont\fontsize{12pt}{12pt}
\begingroup%
  \makeatletter%
  \providecommand\color[2][]{%
    \errmessage{(Inkscape) Color is used for the text in Inkscape, but the package 'color.sty' is not loaded}%
    \renewcommand\color[2][]{}%
  }%
  \providecommand\transparent[1]{%
    \errmessage{(Inkscape) Transparency is used (non-zero) for the text in Inkscape, but the package 'transparent.sty' is not loaded}%
    \renewcommand\transparent[1]{}%
  }%
  \providecommand\rotatebox[2]{#2}%
  \newcommand*\fsize{\dimexpr\f@size pt\relax}%
  \newcommand*\lineheight[1]{\fontsize{\fsize}{#1\fsize}\selectfont}%
  \ifx\svgwidth\undefined%
    \setlength{\unitlength}{258.7668904bp}%
    \ifx\svgscale\undefined%
      \relax%
    \else%
      \setlength{\unitlength}{\unitlength * \real{\svgscale}}%
    \fi%
  \else%
    \setlength{\unitlength}{\svgwidth}%
  \fi%
  \global\let\svgwidth\undefined%
  \global\let\svgscale\undefined%
  \makeatother%
  \begin{picture}(1,0.41337832)%
    \lineheight{1}%
    \setlength\tabcolsep{0pt}%
    \put(0,0){\includegraphics[width=\unitlength,page=1]{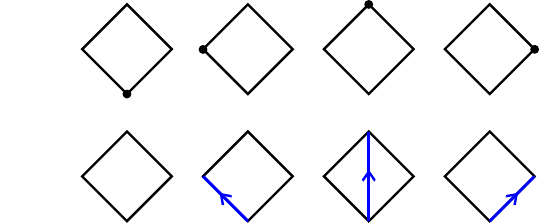}}%
    \put(-0.0023398,0.31115073){\color[rgb]{0,0,0}\makebox(0,0)[lt]{\lineheight{1.25}\smash{\begin{tabular}[t]{l}$\x(S_i)$\end{tabular}}}}%
    \put(0.00925366,0.07348486){\color[rgb]{0,0,0}\makebox(0,0)[lt]{\lineheight{1.25}\smash{\begin{tabular}[t]{l}$e_{\x,i}$\end{tabular}}}}%
  \end{picture}%
\endgroup%

    \caption{Picking the edges $e_{\x,i}$, the union of which is the multi-loop $\mu_\x$ associated to $\x$.}
    \label{fig:statemultiloopdefn}
\end{figure}

\begin{prop}[Alfieri-Tsang {\cite[Proposition 4.12]{AT25a}}]
The map $\x \to \mu_\x$ defines a bijection between the states and the embedded multi-loops of $G_+$.
\end{prop}

We define $\gamma_\x$ to be the closed multi-orbit $\mathcal{F}(\mu_\x)$, where $\mathcal{F}$ is the correspondence between sweep-equivalence classes of loops of $G_+$ and closed orbits of $\phi^\sharp$ in \Cref{thm:pavbscorr}.

\subsection{Ozsváth-Szabó's theory of domains} 

Finally we review very carefully the theory of domains that can be used to understand the differential of Heegaard Floer homology. Since this is very relevant to our discussion in this paper we cover it in full detail.

\begin{defn}[Domains]
Let $(\Sigma, \boa, \bob)$ be a balanced sutured Heegaard diagram. Let $D_1, \dots, D_r$ be the complementary regions of $\boa \cup \bob$ in $\Sigma$ that do not meet $\partial \Sigma$.
We call $D_i$ the \emph{elementary domains} of the diagram.
A \emph{domain} is a formal linear combination $D=\sum_i n_i D_i$  with integer coefficients $n_i \in \Z$. The coefficient $n_i$ is called the \emph{multiplicity} of the domain $D$ at $D_i$.
\end{defn}

A domain $D=\sum_{i=1}^r n_i\ D_i$ is called \emph{effective} if $n_i\geq 0$ for $i=1, \dots, r$. In this case we shall write $D\geq 0$.

\begin{defn}[Connecting domains] \label{defn:connectingdomains}
Let $\x=x_1+ \dots + x_d$ and $\y=y_1+ \dots + y_d$ be two Heegaard states, and $D=\sum_i n_i D_i$ a domain. Thinking of $D$ as a $2$-chain on $\Sigma$, i.e. an element of $C_2(\Sigma; \Z)$, $\partial D$ is a $1$-cycle in the embedded graph $\boa \cup \bob$. Since $C_1(\boa \cup \bob; \Z)= C_1(\boa; \Z) \oplus C_1(\bob; \Z)$, the differential $\partial D$ naturally decomposes as the sum of two $1$-chains $\partial D=\partial_\alpha D + \partial_\beta D$, where $\partial_\alpha D= \partial D \cap \boa$ and $\partial_\beta D= \partial D \cap \bob$. 
We say that the domain $D$ \emph{connects $\x$ to $\y$} if $\partial_\alpha D=\y -\x$ and  $\partial_\beta D= \x -\y$. 
If $D$ connects $\x$ to $\y$ we call $\x$ the initial state of $D$ and $\y$ its final state.
\end{defn} 

We recall the definition of the Lipshitz index of a domain connecting two Heegaard states:
Let $S$ be a surface with corners. The \emph{Euler measure} of $S$ is defined to be 
$$e(S) = \chi(S) - \frac{1}{4} \text{\# corners}$$
where $\chi$ is the Euler characteristic of the underlying surface with boundary (i.e. forgetting the data of the corners).
In particular each elementary domain $D_i$, having a natural structure as a surface with corners, has an Euler measure $e(D_i)$.
The \emph{Euler measure} of a domain $D = \sum_i n_i D_i$ is defined to be $e(D) = \sum_i n_i e(D_i)$.

Suppose $x$ is a point in $\boa \cap \bob$. Suppose $D_i, D_j, D_k, D_l$ are the elementary domains with a corner at $x$. The \emph{average multiplicity} of a domain $D = \sum_i n_i D_i$ at $x$ is defined to be $n_x(D) = \frac{1}{4}(n_i+n_j+n_k+n_l)$.
The \textit{average multiplicity} of a domain $D$ at a state $\x = \{x_1,\dots,x_d\}$ is defined to be $n_\x(D) = \sum_{i=1}^d n_{x_i}(D)$.

\begin{defn} \label{defn:lipshitzindex}
The \emph{Lipshitz index} of a domain $D$ connecting states $\x$ to $\y$ is defined to be 
$$\mu(D) = e(D)+n_\x(D)+n_\y(D).$$
\end{defn}

Associated to a $C^\infty$-disk in $\sym^d(\overline{\Sigma})$ connecting two intersection points of $\T_{\boa}$ and $\T_{\bob}$
\[u: D^2 \to \sym^d(\overline{\Sigma})\setminus  \ \bigcup_{i=1}^k\  \{z_i\} \times \sym^{d-1}(\overline{\Sigma})\]
there is a domain 
\[D(u)= \sum_{i=1}^r\  \#\left( u^{-1} (\{z_i\} \times \sym^{d-1}(\overline{\Sigma})) \right) \cdot D_i \ . \]
We summarize some results about domains from \cite{OS04c} and \cite{Lip06} into the following.

\begin{thm}
The map $u \to D(u)$ has the following properties.
\begin{itemize}
    \item If $u_t$ is a homotopy of $C^\infty$-disks connecting $\x$ to $\y$ then $D(u_0)=D(u_1)$. Vice versa, if $u_0$ and $u_1$ are two disks connecting $\x$ to $\y$ such that $D(u_0)=D(u_1)$ then   $u_0$ and $u_1$ are homotopic through disks connecting $\x$ to $\y$.
    \item If $D$ is a domain connecting $\x$ to $\y$ then there exists a $C^\infty$-disk  in the symmetric product connecting $\x$ to $\y$ with domain $D$.
    \item If $u$ is a $C^\infty$-disk connecting two intersection points  then the Lipshitz index of its domain $D(u)$ equals the Maslov index of $u$.
    \item If $u$ is a holomorphic disk then its associated domain is effective, that is, $D(u)\geq 0$.
\end{itemize}
\end{thm}

In other words there is a one-to-one correspondence between $C^\infty$-disks connecting $\x$ to $\y$ and domains connecting the corresponding Heegaard states in the sense of \Cref{defn:connectingdomains}. Under this correspondence the Maslov index of a disk corresponds to the Lipshitz index of the domain, and holomorphic disks correspond to effective domains.

It is very important to observe here that not all effective domains come from holomorphic disks, and that there can be many holomorphic disks having the same effective domain. Given an effective domain $D$ connecting two intersection points $\x$ and $\y$, we denote by $\mathcal{M}(D)$ the set of all holomorphic disks with domain $D$ up to reparametrization of the source. Then
\[c(\x, \y)=\sum_{D \in \mathcal{P}(\x, \y) } \# \mathcal{M}(D) \ ,  \]
where $\mathcal{P}(\x, \y)$ denotes the set of all effective domains with Lipshitz index one from $\x$ to $\y$. Note that  $\mathcal{P}(\x, \y)= \widetilde{\mathcal{P}}(\x, \y)\cap \Z^r$ for some convex polytope $\widetilde{\mathcal{P}}(\x, \y)\subset \R^r$, and consequently computing $\mathcal{P}(\x, \y)$ is a problem of linear programming. 

Computing the number of points in the moduli space $\mathcal{M}(D) $ can be extremely challenging unless the domain is very specific. 

\begin{defn}[Polygons, {\cite[Definition~3.2]{SW10}}] \label{defn:emptypolygon}
Suppose that $\x$ and $\y$ are two Heegaard states. A \emph{2n-gon} from $\x$ to $\y$ is a domain $D=\sum n_i D_i$ satisfying the following conditions:
\begin{itemize}
\item all multiplicities $n_i$ in $D$ are either 0 or 1,
\item at every coordinate $x_i\in \x$ (and similarly for $y_i \in \y$) either all four domains meeting at $x_i$ have multiplicity 0 (in which case $x_i=y_i$) or exactly one domain has multiplicity 1 and all three others have multiplicity 0 (when $x_i \neq y_i$), and
\item the \emph{support} $s(D)$ of $D$, which is the union of the closures of the elementary domains that have multiplicity, is a subspace of $\Sigma$ which is homeomorphic to the closed disk, with $2n$ vertices on its boundary.
\end{itemize}
The $2n$--gon is \emph{empty} if the interior of $s(D)$ is disjoint from the two given states $\x$ and $\y$. 
\end{defn}

The following is a folklore result in the Heegaard Floer community, see \cite[Proof of Theorem 10.3]{OSS12}. It can be proved by means of the Riemann mapping theorem.

\begin{thm} \label{thm:polygonscountonce}
If $D$ is an empty $2n$-gon then $\#\mathcal{M}(D)=1$.
\end{thm}

\subsection{$\Spinc$-grading} 

For the definition of $\spinc$-structures on sutured manifolds see \cite[Section 3.5]{AT25a}.
In this paper we denote with $\Spinc(M, \Gamma)$ the set of all $\spinc$-structures on a sutured manifold $(M, \Gamma)$. This is an affine space over the group $H^2(M, \partial M;\Z)\simeq H_1(M; \Z)$, in particular there is an identification $\Spinc(M, \Gamma)=H_1(M; \Z)$ once a reference $\spinc$ structure is given. 

In what follows we shall be interested in the case when $M=Y^\sharp$ is the blow-up of a pseudo-Anosov flow $(Y, \phi)$ at a collection of closed orbits $\mathcal{C}$. In this case  there are two canonical $\spinc$-structures:  $\s_{\phi^\sharp}$ induced by the flow, and $\s_{-\phi^\sharp}$ induced by its opposite. Note that these are conjugate to each other, that is $\overline{\s_{\phi^\sharp}}=\s_{-\phi^\sharp}$.

$\Spinc$-structures are relevant in Heegaard Floer homology because given an Heegaard diagram $(\Sigma, \boa, \bob)$ each generator $\x\in \T_{\boa} \cap \T_{\bob}$ has an associated $\spinc$-structure $\s(\x)$. 

\begin{lemma} \label{lemma:obstruction1}
Suppose $\x$ and $\y$ are two Heegaard states. Then $\x$ and $\y$ are connected by a domain if and only if $\s(\x)=\s(\y)$. Consequently, if $\s(\x) \neq \s(\y)$ then $\mathcal{P}(\x, \y)=\varnothing$ and the coefficient $c(\x,\y)$ in the Floer differential vanishes. 
\end{lemma}

It follows that the sutured Floer chain complex $CF(\Sigma, \boa, \bob)$ splits into a direct sum of smaller chain complexes:
\[CF(\Sigma, \boa, \bob) = \bigoplus_{\s \in  \Spinc(M, \gamma)} CF(\Sigma, \boa, \bob, \s) \ ,\]
where $CF(\Sigma, \boa, \bob, \s)= \bigoplus_{\s(\x)=\s } \mathbb{F} \cdot \x$. 

The difference between the $\spinc$-structures associated to two states $\x$ and $\y$ can be expressed in terms of Ozsváth-Szábo's $\epsilon$-map, which is defined as follows. Write $\x=x_1+\dots+x_d$ and $\y = y_1+\dots+y_d$.
Pick paths $a = \alpha_1 \cup \dots \cup \alpha_d$ such that $\partial a=y_1 +\dots+y_d-x_1 -\dots-x_d$ and pick paths $b = \beta_1 \cup \dots \cup \beta_d$ such that $\partial b=y_1 +\dots+y_d-x_1 -\dots-x_d$. Then $\epsilon(x,y)$ is the image of $a-b$ in \[H_1(M)= \frac{H_1(\Sigma)}{ \text{Span}_\Z \langle[\alpha_1], \dots, [\alpha_d],[\beta_1], \dots , [\beta_d]  \rangle } \ .\]

\begin{lemma} \label{lemma:difference}
For each pair of Heegaard states $\x$ and $\y$ we have that \[\s(\x)-\s(\y)= PD[\epsilon(\x,\y)]\] where $PD:  H_1(M; \Z) \to H^2(M, \partial M; \Z)$ denotes Poincar\' e duality.   
\end{lemma}

When $(\Sigma, \boa, \bob)$ is the Heegaard diagram of a veering branched surface there are two preferential intersection points $\xt$, and $\xb$. For these intersection points we showed that $\s(\xt)=\s_{\phi^\sharp}$, and $\s(\xb)=\s_{-\phi^\sharp}= \overline{\s_{\phi^\sharp}}$ in \cite[Proposition 4.19]{AT25a}. 

\begin{lemma}[Alfieri-Tsang {\cite[Proposition 4.18]{AT25a}}] \label{lemma:epsilon=mu=gamma}
If $\x$ is an Heegaard state of the Heegaard diagram associated to a veering branched surface then $\epsilon( \x, \xb)=[\mu_\x] =[\gamma_\x]$ in $H_1(M; \Z)$.
\end{lemma}

Thus if we use $\overline{\s_{\phi^\sharp}}$ as the base $\spinc$-structure, we can identify $\Spinc(Y^\sharp, \Gamma)$ with $H_1(Y^\sharp; \Z)$ and write $\s(\x)=[\gamma_\x]$. Note that with this convention $\s(\xb)=0$.

\section{A refinement of the $\spinc$-grading} \label{sec:wtsgrading}

\subsection{The $\widetilde{\epsilon}$-map}

As we pointed out in \Cref{subsec:heegaardstates}, in the Heegaard diagram $(\Sigma, \boa, \bob)$ of a veering branched surface one can associate to each Heegaard state $\x$ a multi-loop $\mu_\x=\{\mu_1, \dots, \mu_n\}$ in $G_+$. 
For each $i=1, \dots n$, the loop $\mu_i$ either lies in $G\subset G_+$, or it can be strum into a loop $\widetilde{\mu}_i$ in $G$ (recall \Cref{defn:strum}).

Since there is no \emph{canonical} way of strumming each loop $\mu_i$, we shall consider the set of all possible strums
\[\widetilde{\mu}_\x := \Big\{\widetilde{\mu}= \{\widetilde{\mu}_1, \dots, \widetilde{\mu}_n\} \subset G\ \mid \ \widetilde{\mu} \text{ is a strum of the multi-loop  } \mu_\x \Big\} \ .\]
Summarizing: there is no map associating to an Heegaard state a multi-loop in the dual graph $G$, but there is a multi-valued map doing so.

\begin{rmk}
We caution that although $\mu_\x$ is an embedded multi-loop of $G$, the multi-loops in $\widetilde{\mu}_\x$ can be non-embedded in general.    
\end{rmk}

The assignment $\x \mapsto \widetilde{\mu}_\x$ descends in homology to a multi-valued map $\widetilde{\epsilon}: \mathfrak{S}(\Sigma, \boa, \bob) \to H_1(G)$ by setting 
\[\widetilde{\epsilon}(\x) := \left\{[\widetilde{\mu}]=\sum_{i=1}^n [\widetilde{\mu}_i] \in H_1(G)\  \mid \ \widetilde{\mu}= \{\widetilde{\mu}_1, \dots, \widetilde{\mu}_n\} \text{ is a multi-loop in } \widetilde{\mu}_\x \right\} \ .\]
We caution \emph{against} thinking that there is not much difference between $\widetilde{\mu}_\x$ and $\widetilde{\epsilon}(\x)$. Since a given homology cycle can be represented by multiple multi-loops (see \Cref{lemma:loopresolution}), $\widetilde{\mu}_\x$ in fact carries strictly more information than $\widetilde{\epsilon}(\x)$. 

\begin{lemma} \label{lemma:transhomclassliftsepsilon}
The $\widetilde{\epsilon}$-map is a multi-valued lift of the Ozsv\'ath-Szab\' o $\epsilon$-map: 
\begin{center}
   \begin{tikzcd}[column sep=small]
     & H_1(G)\arrow{d}{i_*}  \\
\mathfrak{S}(\Sigma, \boa, \bob) \arrow{ru}{\widetilde{\epsilon}}  
 \arrow{r}{\epsilon} & H_1(M) 
\end{tikzcd} 
\end{center} 
where $i_*:H_1(G)\to H_1(M)$ denotes the map induced by the inclusion. Indeed, given an Heegaard state $\x$,  one has that $i_*(\lambda) =[\gamma_\x] = \epsilon(\x)$ for all  $\lambda \in \widetilde{\epsilon}(\x)$.
\end{lemma}
\begin{proof}
First of all we note that $M$ retracts onto the veering branched surface $B$. Furthermore, the augmented dual graph $G^+$ subdivides a sector $S_i$ into a union of two triangles $\Delta_i^+$, and $\Delta_i^-$ giving $B$ the structure of a simplicial 2-complex.

If $\lambda$ is a strum of $\mu_\x$ then $\lambda \cup \mu_\x $ bounds a union of triangles in this simplicial decomposition. Thus, $\lambda-\mu_\x=\sum_{i=1}^n c_i\cdot \Delta_i^{a_i}$ for some coefficients $c_i\in\{-1, 0, 1\}$, and signs $a_1, \dots, a_n$. It follows from \Cref{lemma:epsilon=mu=gamma} that $[\lambda]=[\mu_\x]=[\gamma_\x] = \epsilon(\x)$ in $H_1(B)=H_1(M)$.
\end{proof}

\subsection{The $\widetilde{\mathfrak{s}}$-grading} 

We now set up a refinement of the $\spinc$-grading based on the multi-valued lift of the Ozsvath-Szabo $\epsilon$-map we defined in the previous subsection. 

We first recall the following facts about effective domains in our Heegaard diagram.

\begin{prop}[{\cite[Proposition 5.9, Lemma 5.13, Proposition 5.15]{AT25a}}] \label{prop:admdomaincorners}
Suppose $(\Sigma,\boa,\bob)$ is a Heegaard diagram associated to a veering branched surface. Let $D$ be an effective domain on $\Sigma$ connecting a Heegaard state $\x$ to a Heegaard state $\y$. 
Then 
\begin{enumerate}
    \item $D$ is embedded, i.e. all the multiplicities are either $0$ or $1$,
    \item the boundary of the support $s(D)$ (recall from \Cref{defn:emptypolygon}) consists of alternating arcs lying on $\alpha$ and $\beta$, i.e. there are no boundary components of $D$ that lie solely on $\alpha$ or $\beta$, and
    \item each $\alpha$- and $\beta$-arc does not contain a point of $\xt$ in its interior.
\end{enumerate}
\end{prop}

Using \Cref{prop:admdomaincorners}, we can prove the following lemma.

\begin{lemma} \label{prop:transhomclassobsdomain}
In the Heegaard diagram of a veering branched surface, if there exists an effective domain $D\geq 0$ connecting a Heegaard state $\x$ to a Heegaard state $\y$ then $\widetilde{\epsilon}(\x) \cap \widetilde{\epsilon}(\y) \neq \varnothing$.
\end{lemma}
\begin{proof}
Write $\x=x_1+ \dots+ x_n$ and $\y=y_1+ \dots+ y_n$ with the convention that $x_i\in \beta_i\cap \alpha_{\sigma(i)}$, and $y_i\in \beta_i\cap \alpha_{\tau(i)}$ for two permutations $\sigma$ and $\tau$.

In $H_1(\Sigma)$ one has that $0= \partial D=\partial_\alpha D + \partial_\beta D$. Thus, modding out by  $H_1(\T_{\boa}) = \text{Span}_\Z\{[\alpha_1], \dots, [\alpha_d]\}$, we conclude that $\partial_\beta D=0$ in $ H_1(\Sigma)/H_1(\T_{\boa})\cong H_1(G)$. Here the last isomorphism can be interpreted as follows: After pinching each $\alpha$-curve to a point, the $\beta$-curves glue together to give a graph $C^0$-close the $1$-skeleton of the branched surface $B$, and the surface $\Sigma$ gives rise to a surface with conic singularities that deformation retracts on said graph.

Since the domain $D$ is embedded (\Cref{prop:admdomaincorners}(1)), the edges involved in $\partial_\beta D$ all come with multiplicity one, so  $\partial_\beta D$ can be decomposed  as  a union of simplicial arcs $e_1, \dots, e_n$ with $e_i\subset \beta_i$, and $\partial e_i=x_i-y_i$. 
Each path $e_i$ gives rise to simplicial arc $\widetilde{e}_i \subset \partial S_i$. 
Using \Cref{prop:admdomaincorners}(3), we check that $\widetilde{e}_i$ can be decomposed as a difference $\widetilde{e}_{\x,i} -\widetilde{e}_{\y,i}$ with $\widetilde{e}_{\x,i}$ a directed path in $G$ connecting $x^\bot_i$ to $x_i$ and $\widetilde{e}_{\y,i}$ being a directed path in $G$ connecting $x^\bot_i$ to $y_i$.
See \Cref{fig:epsilondifference}.
Thus in $Z_1(G)=H_1(G)$ we have that
\[0=\partial_\beta D =\sum_i\  \widetilde{e}_i =\sum_i  \ \widetilde{e}_{\x,i} - \widetilde{e}_{\y,i} \ ,\]
that is
 \[ \sum_i  \widetilde{e}_{\x,i} =\sum_i \widetilde{e}_{\y,i} \ .\] 
It follows that $\widetilde{\epsilon}(\x)$ and $\widetilde{\epsilon}(\y)$  have some multi-loop in common.
\end{proof}

\begin{figure}
    \centering
    \selectfont\fontsize{8pt}{8pt}
    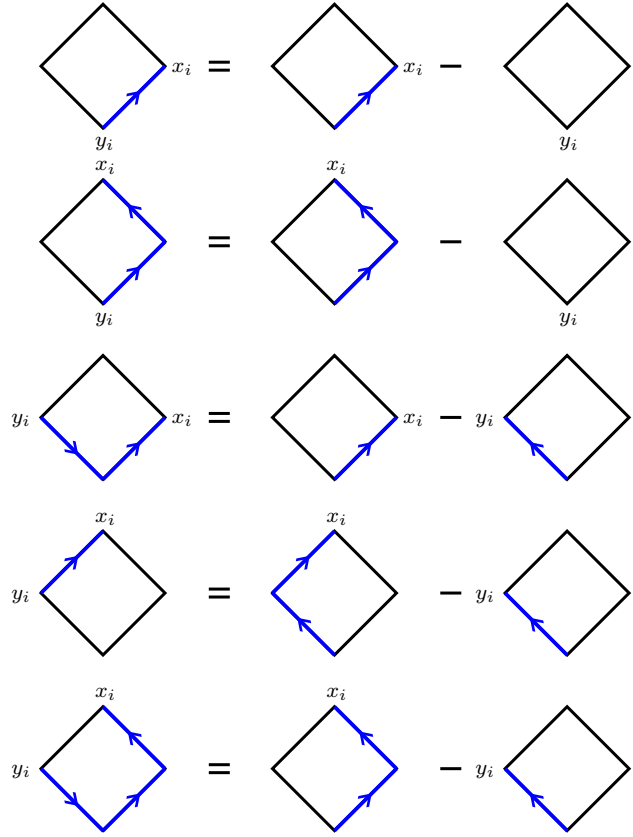
    \caption{Checking that every $\beta$-arc $\widetilde{e}_i$ of $D$, which goes from $y_i$ to $x_i$, can be written as a difference $\widetilde{e}_i = \widetilde{e}_{\x,i} -\widetilde{e}_{\y,i}$. Modulo the symmetry of reflecting the sector across the vertical diagonal, and interchanging $x_i$ and $y_i$, there are 8 cases. In 3 of the cases, $x_i=y_i$ in which case $\widetilde{e}_i$ is the empty arc and the check is trivial. The remaining 5 cases are as shown.}
    \label{fig:epsilondifference}
\end{figure}

The regular $\spinc$-grading can be thought as an equivalence relation where $\x \sim \y$ if and only if $\epsilon(\x)=\epsilon(\y)$. Indeed, by \Cref{lemma:difference}, for two states $\x$ and $\y$, one has that $\s(\x)=\s(\y)$ provided that
\[0= \epsilon(\x, \y)=
\epsilon(\x, \xb) - \epsilon(\y, \xb)=
\epsilon(\x)-\epsilon(\y)
\ . \] 
This equivalence relation partitions Heegaard states  in equivalence classes. If we denote by   $\Spinc(\Sigma, \boa, \bob)$ the set of such equivalence classes, then the map $\s: \gS \hd \to \Spinc(M, \Gamma)$ descends to an injective map $\Spinc(\Sigma, \boa, \bob) \hookrightarrow  \Spinc(M, \Gamma)$, labeling the equivalence classes of $\sim$ with $\spinc$-structures. If a $\spinc$-structure $\s$ does not appear in the image of this map then $SFH(M, \Gamma, \s)=0$. Conversely, if a $\spinc$-structure $\s$ appears in the image of this map the group $SFH(M, \Gamma, \s)$ is the homology of the matrix
\[\partial_\s =[c(\x, \y)]_{\x, \y \in \s} \ ,  \]
where $\x_1, \dots, \x_s$ denote the Heggaard states in the equivalence class labeled by $\s$. In fact, as a consequence of \Cref{lemma:obstruction1}, the matrix of the Heegaard Floer differential has a block decomposition
 \[\partial = \begin{bmatrix} \partial_{\s_1} & & \\ & \ddots & \\ & & \partial_{\s_s} \end{bmatrix} \]
 where $\s_1, \dots, \s_s$ denote the $\spinc$-structures hit by the $\s$-map. 
 
\Cref{prop:transhomclassobsdomain} suggests that each sub-matrix $\partial_\s$ can be further decomposed into smaller blocks. In order to describe this partition of the basis we  would like to declare two Heegaard states $\x$ and $\y$ to be equivalent if $\e(\x)\cap \e(\y)\neq \varnothing$. This is a reflexive  ($\e(\x)\cap \e(\x)\neq \varnothing$), and symmetric relation ($\e(\x)\cap \e(\y)=\e(\y)\cap \e(\x)$), but unfortunately it is \emph{not} a transitive relation: It can be that $\e(\x)\cap \e(\y)\neq \varnothing$, and  $\e(\y)\cap \e(\z)\neq \varnothing$ but that $\e(\x)\cap \e(\z)$ is empty. For this reason we have to consider the equivalence relation \emph{generated by} $\x\sim \y$ if $\e(\x)\cap \e(\y)\neq \varnothing$.  

\begin{defn}[$\widetilde{\epsilon}$-equivalence of Heegaard states] \label{defn:tsequivalence}
We declare two states $\x$ and $\y$ to be \textbf{$\widetilde{\epsilon}$-equivalent} if there exists a sequence of states $\x = \z_0, \dots, \z_n = \y$ such that 
$$\e(\z_i)\cap \e(\z_{i+1}) \neq \varnothing$$ 
for $i=0, \dots, n-1$.We define the \textbf{$\widetilde{\s}$-grading} of a state $\x$ to be the $\widetilde{\epsilon}$-equivalence class that contains $\x$, and denote it by $\ts(\x)$.
\end{defn}

In the following, we shall denote by $\widetilde{\Spinc}\hd$ the set of equivalence classes of the $\widetilde{\epsilon}$-equivalence relation.
\begin{prop}
The relation of $\epsilon$-equivalence is coarser than the relation of  $\widetilde{\epsilon}$-equivalence. Consequently, there is a map $\pi: \widetilde{\Spinc}\hd \to \Spinc \hd$. 
\end{prop}
\begin{proof}
If $\x$ and $\y$ are $\e$-equivalent there is a sequence $\x = \z_0, \dots, \z_n = \y$ of Heegaard states and a sequence of multi-loops $c_1, \dots , c_n$ with $c_i\in \e(\z_i) \cap \e(\z_{i+1})$, for $i=1, \dots, n$. By \Cref{lemma:transhomclassliftsepsilon}, we have that 
\[[\gamma_\x]=[\gamma_{\z_{0}}]=i_*[c_0]= [\gamma_{\z_{1}}]=i_*[c_1]=[\gamma_{\z_{2}}] = \dots = i_*[c_{n-1}]=[\gamma_{\z_{n}}] = [\gamma_\y]\] 
where $i_*$ denotes the map induced by the inclusion $G \to M$. Thus  we have that:
\[\s(\x)-\s(\y)= \epsilon(\x,\y)=\epsilon(\x)-\epsilon(\y)=[\gamma_\x]-[\gamma_\y]=0, \]
proving that $\x$ and $\y$ have the same $\spinc$-structure.
\end{proof}

Summarizing, we have maps
\begin{center}
   \begin{tikzcd}[column sep=small]
\mathfrak{S}(\Sigma, \boa, \bob) \arrow{r}{\widetilde{\s}}  
 & \widetilde{\Spinc}\hd  \arrow{r}{\pi} &      \Spinc\hd \hookrightarrow  \Spinc(M, \Gamma) \ ,
\end{tikzcd} 
\end{center} 
whose composition agrees with the traditional $\spinc$-grading map of Ozsváth and Szabó. 
It is unknown to the authors if the $\ts$-grading has an interpretation in terms of geometric structures as the $\spinc$ grading. 

\begin{cor}
Let $\x$ and $\y$ be two Heegaard states in the canonical Heegaard diagram of a veering branched surface. If $\ts(\x)\neq\ts(\y)$ then in the coefficient $c(\x,\y)$ in the Heegaard Floer differential vanish. \qed
\end{cor}

Given $\widetilde{\mathfrak{s}}\in \widetilde{\Spinc}\hd$ we denote by 
\[SFC(\phi,\mathcal{C},\widetilde{\mathfrak{s}}) = \bigoplus_{\x \in \widetilde{\mathfrak{s}}} \mathbb{F} \cdot \x\] the subcomplex of $SFC(\phi,\mathcal{C})$ spanned by the Heegaard states with $\ts$-grading $\ts$. For a given $\spinc$-structure $\s\in \Spinc \hd $, we have that
\[SFC(\phi,\mathcal{C},\s) = \bigoplus_{\pi(\ts)= \s } SFC(\phi,\mathcal{C},\widetilde{\mathfrak{s}})\]
where $\pi$ denotes the projection $\widetilde{\Spinc}\hd \to \Spinc \hd$.

\section{Relation with Fried pants} \label{sec:friedpants}

In the chain complex $SFC(\phi, \mathcal{C})$ the $\spinc$-grading can be interpreted as the map associating to a generator $\x$ the homology class $[\gamma_\x]\in H_1(Y^\sharp)$ of the multi-orbit $\gamma_\x$ attached to $\x$.
An obvious refinement of this would be to consider the multi-orbit $\gamma_\x$ itself. By \Cref{thm:pavbscorr}, this is equivalent to the sweep-equivalence class of the multi-loop $\mu_\x \subset G_+$.

Note that given two Heegaard states $\x$ and $\y$, the associated multi-loops $\mu_\x$ and $\mu_\y$ are sweep-equivalent if and only if there is a sequence  of states $\x = \z_0, \dots, \z_n = \y$  such that $\widetilde{\mu}_{\z_{i-1}}\cap \widetilde{\mu}_{\z_{i}} \neq \varnothing$ for $i=1, \dots, n$. 

In other words, the difference between considering the multi-orbit $\gamma_\x$ and the refined $\spinc$-grading $\ts(\x)$ is exactly the difference between stating \Cref{defn:tsequivalence} using $\widetilde{\mu}_{\x}$ or $\e(\x)$. In this section we explore the difference between these two concepts since the first is of geometric relevance in the dynamical set-up, while the second plays a special role in the Heegaard Floer chain complex.

\subsection{Sleekness}
We start by discussing a specific situation in which there is no difference between $\gamma_\x$ and the refined $\ts$-grading.

\begin{defn}[Sleek] \label{defn:sleek}
We shall make use of the following terminology:
\begin{itemize}
    \item We say that a multi-loop $c$ in the augmented dual graph $G_+$ is \emph{sleek} if every multi-loop $c'$ that is sweep-equivalent to $c$ is embedded.
    \item We say that a multi-orbit $\gamma$ is \emph{sleek} if $\gamma=\mathcal{F}(c)$ for some sleek multi-loop $c$ (recall $\mathcal{F}$ from \Cref{thm:pavbscorr}). 
    \item We say that a Heegaard state $\x$ is \emph{sleek} if its associated multi-orbit $\gamma_\x$ is sleek.
    \item We say that a class $\widetilde{\mathfrak{s}}\in\widetilde{\Spinc}\hd$ is \emph{sleek} if it contains some sleek state $\x$.
\end{itemize}
\end{defn}

The motivation of this terminology comes from the mental image that as one sweeps $c$ around, it never gets tangled up with itself.

\begin{eg} \label{prop:embbranchloopssleek}
Any embedded multi-loop that is a union of branch loops is sleek. This is because they are embedded themselves, and they cannot be swept around.
\end{eg}

We show that \Cref{defn:sleek} can be rephrased in the following ways.

\begin{lemma} \label{lemma:sleekorbit}
If $\gamma$ is a sleek multi-orbit then every multi-loop $c$ such that $\gamma=\mathcal{F}(c)$ is sleek.
\end{lemma}
\begin{proof}
By definition, if $\gamma$ is sleek then one can find a sleek multi-loop $c'$ such that $\gamma=\mathcal{F}(c')$.
\Cref{thm:pavbscorr} implies that $c$ and $c'$ are sweep-equivalent, thus $c$ is sleek as well.
\end{proof}

\begin{lemma} \label{lemma:sleekgrading}
If $\widetilde{\mathfrak{s}}\in\widetilde{\Spinc}\hd$ is sleek  then every state $\x$ in $\widetilde{\mathfrak{s}}$ is sleek. Furthermore, there exists a multi-orbit  of the pseudo Anosov flow $\gamma_{\ts}$  such that $\gamma_\x=\gamma_{\ts}$ for all  $\x\in\ts$.
\end{lemma}
\begin{proof}
We have  to show that if $\x$ and $\y$ are Heegaard states, $\x$ is sleek, and $\e(\x) \cap \e(\y) \neq \varnothing$, then $\y$ is also sleek and $\gamma_\x = \gamma_\y$. So suppose $c_\x \in \widetilde{\mu}_\x$ and $c_\y \in \widetilde{\mu}_\y$ are two multi-loop such that $[c_\x] = [c_\y]$ in $H_1(G)$. 

By definition, $\x$ being sleek means that the multi-orbit $\gamma_\x$ associated to $\x$ is sleek. 
Thus by \Cref{lemma:sleekorbit} every multi-loop in $\widetilde{\mu}_\x$ is embedded, showing that $c_\x$ is embedded. 

Note that in general, for a directed graph $G$, if a homology cycle $\eta$ is the homology class of an \emph{embedded} multi-loop $c$, then $\eta$ cannot be the homology class of any other multi-loop. Consequently, $c_\y = c_\x$ showing that also $c_\y$ is embedded, and that the multi-orbits $\gamma_\x$ and $\gamma_\y$ are equal.
\end{proof}

If $\ts \in \widetilde{\Spinc}\hd$ is sleek we say that the associated summand $SFC(\phi,\mathcal{C},\widetilde{\mathfrak{s}})$ is \emph{sleek}. It follows that sleek summands of $SFC(\phi,\mathcal{C})$
are generated by sleek states.

\begin{prop} \label{prop:sleekclasssleekorbit}
If $\x$ and $\y$ are sleek states then $\gamma_\x= \gamma_\y$ if and only if $\ts(\x)=\ts(\y)$.
\end{prop}
\begin{proof}
If $\widetilde{\mathfrak{s}}(\x)=\widetilde{\mathfrak{s}}(\y)$ then $\gamma_\x = \gamma_{\widetilde{\mathfrak{s}}(\x)}=\gamma_{\widetilde{\mathfrak{s}}(\y)} = \gamma_\y$ by \Cref{lemma:sleekgrading}. 
Conversely, if $\gamma_{\x} = \gamma_{\y}$ then $\mu_{\x}$ and $\mu_{\y}$ are sweep-equivalent by \Cref{thm:pavbscorr}. That is, there exists  multi-loops $\mu_{\x} = c_0,\dots,c_n = \mu_{\y}$ in the augmented dual graph $G_+$ such that $c_{i-1}$ and $c_i$ are related by strumming.
Since $\mu_{\x}$ is sleek, each $c_i$ is embedded. Thus by \Cref{prop:generators} they each determine a state $\y_i$, with $\e(\y_{i-1}) \cap \e(\y_{i}) \neq \varnothing$. 
Thus $\x$ and $\y$ lie in the same $\ts$-grading.
\end{proof}

\subsection{Pants resolutions}

In this subsection, we explain a general  way of translating the loss of information between the multi-orbits and the homology cycles in terms of the flow.

\begin{defn} \label{defn:loopjoin}
Let $c_1 = (e_{1,i})_{i \in \mathbb{Z}/N_1}$ and $c_2 = (e_{2,i})_{i \in \mathbb{Z}/N_2}$ be two loops in $G$. Suppose the terminal vertex of $e_{1,N_1}$ and the initial vertex of $e_{2,1}$ are the same vertex $v$ of $G$. Then we can form the concatenation $c_1 * c_2 = (e_{1,1},\dots,e_{1,N_1},e_{2,1},\dots,e_{2,N_2})$. 
\end{defn}

Suppose that $c=\{c_1, c_2,c_3, \dots, c_n\}$ is a multi-loop in the dual graph, and that $c_1$ and $c_2$ are as in \Cref{defn:loopjoin}.
Then we say that $c$ and $c'=\{c_1*c_2,c_3 \dots , c_n\}$ are related by \emph{resolution} at $v$.

\begin{lemma} \label{lemma:loopresolution}
Let $\mu$ and $\mu'$ be two multi-loops in $G$. Suppose $[\mu] = [\mu']$ in $H_1(G)$, then there exists multi-loops $\mu=\mu_0, \dots, \mu_n = \mu'$ such that $\mu_i$ and $\mu_{i+1}$ are connected by resolution at some vertex for each $i$.
\end{lemma}
\begin{proof}
Let $\alpha$ be the common cycle $[\mu] = [\mu']$. 
Then each multi-loop $c$ whose homology class is $\alpha$ can be constructed by taking $\alpha(e)$ copies of each directed edge $e$ and connecting their endpoints in some way at each vertex $v$.
More precisely, indexing the $\alpha(e)$ copies of $e$ as $e_1,\dots,e_{\alpha(e)}$, we have a bijection 
$$\bigsqcup_{\text{$e$ incoming at $v$}} \{e_1,\dots,e_{\alpha(e)}\} \longleftrightarrow \bigsqcup_{\text{$e$ outgoing at $v$}} \{e_1,\dots,e_{\alpha(e)}\}$$
at each vertex $v$.
Composing the bijection at $v$ by a transposition corresponds to performing a resolution at $v$.
Since any two bijections are related by a sequence of transpositions, any two multi-loops representing $\alpha$ are related by a sequence of resolutions. 
\end{proof}

We then introduce the analogous resolution operation for orbits.

\begin{defn}
An \emph{(immersed) Fried pants} of $\phi^\sharp$ is an oriented immersed surface $P \looparrowright M$ where:
\begin{itemize}
    \item $P$ is homeomorphic to a pair of pants.
    \item The interior of $P$ is positively transverse to $\phi^\sharp$.
    \item Each boundary component of $P$ is a closed orbit.
\end{itemize}
See \Cref{fig:friedpants}.
Each boundary component of $P$ inherits an orientation from the direction of the flow and an orientation from the orientation on $P$.
We say that a boundary component is \emph{positive} if these two orientations agree, otherwise we say that the boundary is \emph{negative}.
\end{defn}

One can think about Fried pants as some sort of surgery operation on closed orbits. We use Fried pants  to introduce an equivalence relation on  multi-orbits.

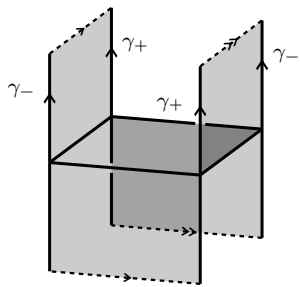
\begin{figure}
    \centering
    \selectfont\fontsize{8pt}{8pt}
\begingroup%
  \makeatletter%
  \providecommand\color[2][]{%
    \errmessage{(Inkscape) Color is used for the text in Inkscape, but the package 'color.sty' is not loaded}%
    \renewcommand\color[2][]{}%
  }%
  \providecommand\transparent[1]{%
    \errmessage{(Inkscape) Transparency is used (non-zero) for the text in Inkscape, but the package 'transparent.sty' is not loaded}%
    \renewcommand\transparent[1]{}%
  }%
  \providecommand\rotatebox[2]{#2}%
  \newcommand*\fsize{\dimexpr\f@size pt\relax}%
  \newcommand*\lineheight[1]{\fontsize{\fsize}{#1\fsize}\selectfont}%
  \ifx\svgwidth\undefined%
    \setlength{\unitlength}{112.81716534bp}%
    \ifx\svgscale\undefined%
      \relax%
    \else%
      \setlength{\unitlength}{\unitlength * \real{\svgscale}}%
    \fi%
  \else%
    \setlength{\unitlength}{\svgwidth}%
  \fi%
  \global\let\svgwidth\undefined%
  \global\let\svgscale\undefined%
  \makeatother%
  \begin{picture}(1,0.92895679)%
    \lineheight{1}%
    \setlength\tabcolsep{0pt}%
    \put(0,0){\includegraphics[width=\unitlength,page=1]{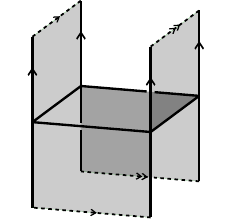}}%
    \put(0.38001232,0.79554494){\color[rgb]{0,0,0}\makebox(0,0)[lt]{\lineheight{1.25}\smash{\begin{tabular}[t]{l}$\gamma_+$\end{tabular}}}}%
    \put(0.49596537,0.59626631){\color[rgb]{0,0,0}\makebox(0,0)[lt]{\lineheight{1.25}\smash{\begin{tabular}[t]{l}$\gamma_+$\end{tabular}}}}%
    \put(0.88449212,0.75840244){\color[rgb]{0,0,0}\makebox(0,0)[lt]{\lineheight{1.25}\smash{\begin{tabular}[t]{l}$\gamma_-$\end{tabular}}}}%
    \put(-0.00134173,0.64110963){\color[rgb]{0,0,0}\makebox(0,0)[lt]{\lineheight{1.25}\smash{\begin{tabular}[t]{l}$\gamma_-$\end{tabular}}}}%
  \end{picture}%
\endgroup%

    \caption{A Fried pants}
    \label{fig:friedpants}
\end{figure}

\begin{defn}[Pants-equivalence]
Let $P$ be a Fried pants. Let $\gamma_\pm$ be the closed multi-orbit that is the collection of positive/negative boundary components of $P$. For every closed multi-orbit $\gamma^\circ$, we say that the closed multi-orbits $\gamma^\circ \cup \{G\}$ and $\gamma^\circ \cup \{\gamma_-\}$ are related by \emph{resolution} along $P$. 

We say that two multi-orbits $\gamma$ and $\gamma'$ are \emph{pants-equivalent} if there exists a sequence of multi-orbits $\gamma = \gamma_0, \dots, \gamma_n = \gamma'$ such that $\gamma_i$ and $\gamma_{i+1}$ are related by resolution along a Fried pants.
\end{defn}

We now prove one of the main technical lemmas of the paper.

\begin{lemma} \label{prop:loopresolveimplyorbitresolve}
Suppose multi-loops $c$ and $c'$ of $G$ are related by a resolution at a vertex, then the corresponding multi-orbits $\mathcal{F}(c)$ and $\mathcal{F}(c')$ are related by resolution along a Fried pants.
\end{lemma} 
\begin{proof}
For this proof, we need to recall the definition of the map $\mathcal{F}$ between loops of $G$ and closed orbits of $\phi^\sharp$.

Lift $B$ to a branched surface $\widetilde{B}$ with branch locus $\widetilde{G}$ in the universal cover $\widetilde{M}$.
Each vertex $\widetilde{v}_0$ of $\widetilde{G}$ corresponds to a \emph{maximal rectangle} $R_{\widetilde{v}_0}$ in the orbit space $\mathcal{O}$ of $\phi$. This is a rectangle embedded in $\mathcal{O}$ where the restrictions of the stable/unstable foliations $\mathcal{O}^{s/u}$ foliate $R_{\widetilde{v}_0}$ as a product, such that there are no elements of $\widetilde{\mathcal{C}}$ in the interior of $R_{\widetilde{v}_0}$, but there is an element of $\widetilde{\mathcal{C}}$ in the interior of each edge of $R_{\widetilde{v}_0}$.

Furthermore, if $\widetilde{v}_1$ and $\widetilde{v}_2$ are the two vertices that follow $\widetilde{v}_0$ in $\widetilde{G}$, then $R_{\widetilde{v}_1}$ and $R_{\widetilde{v}_2}$ are the two maximal rectangles sharing three of the four points of $\widetilde{\mathcal{C}}$ on the boundary as $R_{\widetilde{v}_0}$, including the two that lie on the $\mathcal{O}^s$-sides of $R_{\widetilde{v}_0}$. 
See \Cref{fig:maxrectdualgraph}. 
In particular each $R_{\widetilde{v}_i}$ is taller and thinner than $R_{\widetilde{v}_0}$, in the sense that every $\mathcal{O}^s$-leaf that intersects $R_{\widetilde{v}_i}$ also intersects $R_{\widetilde{v}_0}$ and every $\mathcal{O}^u$-leaf that intersects $R_{\widetilde{v}_0}$ also intersects $R_{\widetilde{v}_i}$.

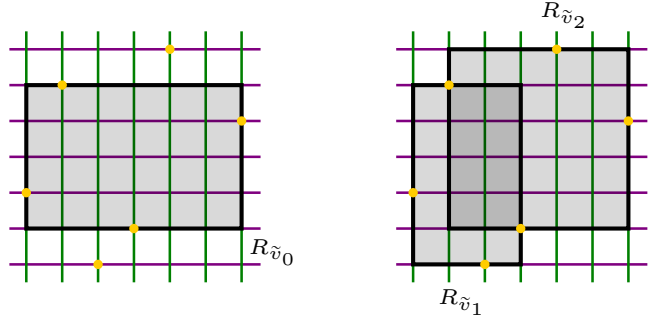
\begin{figure}
    \centering
    \selectfont\fontsize{6pt}{6pt}
    \resizebox{!}{4cm}{
\begingroup%
  \makeatletter%
  \providecommand\color[2][]{%
    \errmessage{(Inkscape) Color is used for the text in Inkscape, but the package 'color.sty' is not loaded}%
    \renewcommand\color[2][]{}%
  }%
  \providecommand\transparent[1]{%
    \errmessage{(Inkscape) Transparency is used (non-zero) for the text in Inkscape, but the package 'transparent.sty' is not loaded}%
    \renewcommand\transparent[1]{}%
  }%
  \providecommand\rotatebox[2]{#2}%
  \newcommand*\fsize{\dimexpr\f@size pt\relax}%
  \newcommand*\lineheight[1]{\fontsize{\fsize}{#1\fsize}\selectfont}%
  \ifx\svgwidth\undefined%
    \setlength{\unitlength}{181.17328794bp}%
    \ifx\svgscale\undefined%
      \relax%
    \else%
      \setlength{\unitlength}{\unitlength * \real{\svgscale}}%
    \fi%
  \else%
    \setlength{\unitlength}{\svgwidth}%
  \fi%
  \global\let\svgwidth\undefined%
  \global\let\svgscale\undefined%
  \makeatother%
  \begin{picture}(1,0.45923758)%
    \lineheight{1}%
    \setlength\tabcolsep{0pt}%
    \put(0,0){\includegraphics[width=\unitlength,page=1]{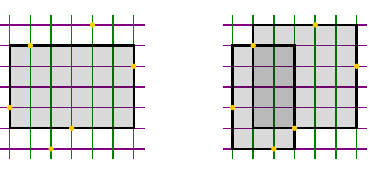}}%
    \put(0.36535996,0.08128249){\color[rgb]{0,0,0}\makebox(0,0)[lt]{\lineheight{1.25}\smash{\begin{tabular}[t]{l}$R_{\wt{v}_0}$\end{tabular}}}}%
    \put(0.65418333,0.00422591){\color[rgb]{0,0,0}\makebox(0,0)[lt]{\lineheight{1.25}\smash{\begin{tabular}[t]{l}$R_{\wt{v}_1}$\end{tabular}}}}%
    \put(0.81008133,0.44224764){\color[rgb]{0,0,0}\makebox(0,0)[lt]{\lineheight{1.25}\smash{\begin{tabular}[t]{l}$R_{\wt{v}_2}$\end{tabular}}}}%
  \end{picture}%
\endgroup%
}
    \caption{Left: Each vertex $\widetilde{v}_0$ of $\widetilde{G}$ gives a maximal rectangle $R_{\widetilde{v}_0}$. Right: If $\widetilde{v}_1$ and $\widetilde{v}_2$ are the two vertices that follow $\widetilde{v}_0$ in $\widetilde{G}$, then $R_{\widetilde{v}_1}$ and $R_{\widetilde{v}_2}$ are taller and thinner than $R_{\widetilde{v}_0}$.}
    \label{fig:maxrectdualgraph}
\end{figure}

Given a loop $c$, lift it to a directed bi-infinite path $\widetilde{c}$ in $\widetilde{G}$. Take the sequence of vertices $(\widetilde{v}_n)$ passed through by this line. The intersection of rectangles $\bigcap R_{\widetilde{v}_n}$ is a point, corresponding to an orbit $\widetilde{\gamma}$ of $\widetilde{\phi}$. One can show that the image of $\widetilde{\gamma}$ in $M$ is a closed orbit $\gamma$ of homotopy class $[c]$.
We refer the reader to \cite{LMT23} for more details.
One defines $\mathcal{F}(c)$ to be $\gamma$.

Now suppose $c$ and $c'$ are two loops in $G$ passing through the same vertex $v$. Lift $c$ and $c'$ to bi-infinite paths $\widetilde{c}$ and $\widetilde{c}'$ passing through a lift $\widetilde{v}$ of $v$. 
Let $\widetilde{v}_0, \dots, \widetilde{v}_n = \widetilde{v}$ be the vertices passed through by one fundamental domain of $\widetilde{c}$, and let $\widetilde{v} = \widetilde{v}'_0, \dots, \widetilde{v}'_{n'}$ be the vertices passed through by one fundamental domain of $\widetilde{c}'$.
In particular, we have $[c] \cdot \widetilde{v}_0 = \widetilde{v}_n = \widetilde{v}$ and $[c'] \cdot \widetilde{v} = [c'] \cdot \widetilde{v}'_0 = \widetilde{v}'_{n'}$.

For each $i=0,\dots,n$, lift the rectangle $R_{\widetilde{v}_i} \subset \mathcal{O}$ an embedded rectangle $\widetilde{R}_{\widetilde{v}_i} \subset \widetilde{M}$ positively transverse to $\widetilde{\phi}$ and intersecting $\widetilde{\gamma}$.
Similarly, for each $i=0,\dots,n'$, lift the rectangle $R_{\widetilde{v}'_i} \subset \mathcal{O}$ an embedded rectangle $\widetilde{R}_{\widetilde{v}'_i} \subset \widetilde{M}$ positively transverse to $\widetilde{\phi}$ and intersecting $\widetilde{\gamma}'$.
Furthermore, we choose the lifts so that $[c] \cdot \widetilde{R}_{\widetilde{v}_0} = \widetilde{R}_{\widetilde{v}_n} = \widetilde{R}_{\widetilde{v}}$ and $[c'] \cdot \widetilde{R}_{\widetilde{v}} = [c'] \cdot \widetilde{R}_{\widetilde{v}'_0} = \widetilde{R}_{\widetilde{v}'_{n'}}$.
In particular, the images of $\widetilde{R}_{\widetilde{v}_0}, \widetilde{R}_{\widetilde{v}}$, and $\widetilde{R}_{\widetilde{v}'_{n'}}$ in $M$ are the same immersed rectangle $R$ that is positively transversely to $\phi$ and intersecting the closed orbits $\gamma$ and $\gamma'$.
See \Cref{fig:friedpantsarg}.

We can now perform Fried's construction in \cite{Fri83}:
For each $i$, following the orbits of $\widetilde{\phi}$ upwards determines a partially defined map $\widetilde{R}_{\widetilde{v}_{i-1}} \dashrightarrow \widetilde{R}_{\widetilde{v}_{i}}$.
More specifically, the map is defined on a subrectangle of $\widetilde{R}_{\widetilde{v}_{i-1}}$ that is as tall as $\widetilde{R}_{\widetilde{v}_{i-1}}$, i.e. every $\widetilde{\Lambda^u}$-leaf that intersects $\widetilde{R}_{\widetilde{v}_{i-1}}$ intersects this domain subrectangle, mapping it homeomorphically to a subrectangle of $\widetilde{R}_{\widetilde{v}_{i}}$ that is as wide as $\widetilde{R}_{\widetilde{v}_{i}}$, i.e. every $\widetilde{\Lambda^s}$-leaf that intersects $\widetilde{R}_{\widetilde{v}_{i}}$ intersects this image subrectangle.

\begin{figure}
    \centering
    \selectfont\fontsize{6pt}{6pt}
    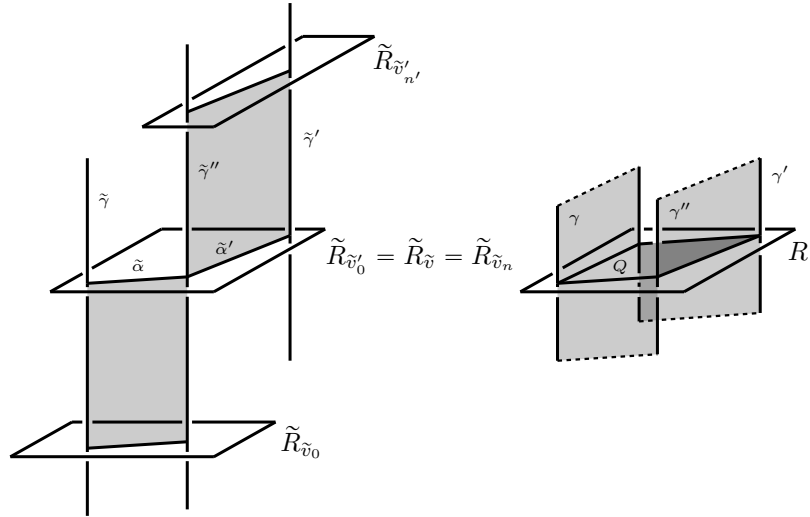
    \caption{Constructing a Fried pants in the proof of \Cref{prop:loopresolveimplyorbitresolve}.}
    \label{fig:friedpantsarg}
\end{figure}

The composition of all these partially defined maps with the deck transformation $[c*c']$ gives a partially defined map $F: \widetilde{R}_{\widetilde{v}_0} \dashrightarrow \widetilde{R}_{\widetilde{v}_0}$.
As for the individual partially defined maps, $F$ maps a subrectangle of $\widetilde{R}_{\widetilde{v}_0}$ that is as tall as $\widetilde{R}_{\widetilde{v}_0}$ to a subrectangle of $\widetilde{R}_{\widetilde{v}_0}$ that is as wide as $\widetilde{R}_{\widetilde{v}_0}$. Thus there is at least one fixed point. We pick one such point and suppose it is the intersection of an orbit $\widetilde{\gamma}''$ of $\widetilde{\phi}$ with $\widetilde{R}_{\widetilde{v}_0}$.

Note that the image $\gamma''$ of $\widetilde{\gamma}''$ is the closed orbit corresponding to the loop $c*c'$.
This is because $\widetilde{\gamma}''$ passes through the rectangles $\widetilde{R}_{\widetilde{v}_0}, \dots, \widetilde{R}_{\widetilde{v}_n} = \widetilde{R}_{\widetilde{v}} = \widetilde{R}_{\widetilde{v}'_0}, \dots, \widetilde{R}_{\widetilde{v}'_{n'}}$.

Now pick an arc $\widetilde{\alpha}$ on $\widetilde{R}_{\widetilde{v}}$ connecting $\widetilde{\gamma} \cap \widetilde{R}_{\widetilde{v}}$ to $\widetilde{\gamma}'' \cap \widetilde{R}_{\widetilde{v}}$ that is transverse to $\widetilde{\Lambda^s}$ and $\widetilde{\Lambda^u}$. 
Similarly, pick an arc $\widetilde{\alpha}'$ on $\widetilde{R}_{\widetilde{v}}$ connecting $\widetilde{\gamma}' \cap \widetilde{R}_{\widetilde{v}}$ to $\widetilde{\gamma}'' \cap \widetilde{R}_{\widetilde{v}}$ that is transverse to $\widetilde{\Lambda^s}$ and $\widetilde{\Lambda^u}$. 

Consider the projection to $M$ of the rectangle $\widetilde{\alpha} \times [-1,0]$ obtained by taking the union of (backward) flow segments from $\widetilde{\alpha}$ to $\widetilde{R}_{\widetilde{v}_0}$ and the rectangle $\widetilde{\alpha'} \times [0,1]$ obtained by taking the union of flow segments from $\widetilde{\alpha'}$ to $\widetilde{R}_{\widetilde{v}'_{n'}}$.
This gives two immersed rectangles tangent to the flow $\phi^\sharp$, each having a top side and a bottom side on $R$. The four sides bound a quadrilateral $Q \subset R$ with corners at $\gamma \cap R$, $\gamma' \cap R$, and two points of $\gamma'' \cap R$.
See \Cref{fig:friedpantsarg} again.

Perturbing the union of $Q$ and the two immersed rectangles to be transverse to the flow, we get a Fried pants $P$ so that $\gamma \cup \gamma'$ is related to $\gamma''$ by resolution along $P$.
This concludes the proof.
\end{proof}

As an immediate corollary  of \Cref{lemma:loopresolution} we have the following theorem  generalizing \Cref{prop:sleekclasssleekorbit} from the previous subsection.

\begin{thm}
Suppose $\x$ and $\y$ are Heegaard states such that $\ts(\x)=\ts(\y)$. Then $\gamma_\x$ and $\gamma_\y$ are pants-equivalent.
\end{thm}

Finally we put the discussion of this subsection in context with the discussion regarding sleek states.  We say that a closed multi-orbit is \emph{pants-irreducible} if it is not connected to any closed multi-orbit by resolving along a Fried pants. 

\begin{prop} \label{prop:pantsirreducibleimplysleek}
A pants-irreducible multi-orbit is sleek.
\end{prop}
\begin{proof}
Suppose $\gamma$ is a pants-irreducible multi-orbit. There is some multi-loop $c_0$ such that $\gamma = \mathcal{F}(c_0)$. 
We have to show that $c_0$ is sleek, that is, we need to show that if $c$ is sweep-equivalent to $c_0$, then $c$ is embedded.

For such a $c$, we have $\gamma = \mathcal{F}(c)$. If $c$ is not embedded, then $c$ can be resolved at a vertex. But then by \Cref{prop:loopresolveimplyorbitresolve}, this implies that $\gamma$ can be resolved along a Fried pants, contradicting pants-irreducibility.
\end{proof}

\Cref{prop:pantsirreducibleimplysleek} is sometimes convenient because checking for sleekness involves dealing with the veering branched surface and involves checking for all multi-loops in a sweep-equivalence class, which is difficult, whereas checking for pants-irreducibility just involves studying the flow.

\section{Dynamic annuli and Möbius bands} \label{sec:dynannuli}

In this section, we develop the language of dynamic annuli and Möbius bands. These provide a combinatorial device for keeping track of the set of loops in $G_+$ that represent a closed orbit. 

\subsection{Construction}

We start by recalling the language of dynamic planes. These were introduced by Landry, Minsky, and Taylor in \cite{LMT23}.

Let $B$ denote as usual a veering branched surface in a 3-manifold $M$, $G$ its dual graph, and $G_+$ its augmented dual graph. Consider $\widetilde{M}$, the universal cover of $M$, and $\widetilde{B}$, the lift of $B$ to $\widetilde{M}$. Note that $\widetilde{B}$ is also a branched surface with branch locus $\widetilde{G}$. We say that a path $\rho : I \to\widetilde{B}$ is \emph{positive} if it intersects $\widetilde{G}$ transversely, and at the points where $\rho(t)\in \widetilde{G}$, one has that $\dot{\rho}(t)$ points in the direction prescribed by the maw co-orientation.

\begin{defn}[Dynamic plane]
Let $c$ be a loop in $G_+$. We choose a lift of $c$ to a bi-infinite path $\widetilde{c}: \R \to \widetilde{G_+}$. The \emph{dynamic plane} associated to $c$, which we denote by $\widetilde{D}(c)$, is the union of all sectors $S$ of $\widetilde{B}$ for which there is a positive path starting at a point on $\widetilde{c}$ and ending inside $S$.
\end{defn}

The arguments in \cite[Section 3.1]{LMT23} show that if $c$ does not cover a branch loop, then $\widetilde{D}(c)$ is homeomorphic to a plane. We illustrate an example in \Cref{fig:dynplane}. In the special case when $c$ covers a branch loop, $\widetilde{D}(c)$ is homeomorphic to a half-plane, and has as a boundary line the lift $\widetilde{c}$.

\begin{figure}
    \centering
    \selectfont\fontsize{10pt}{10pt}
\begingroup%
  \makeatletter%
  \providecommand\color[2][]{%
    \errmessage{(Inkscape) Color is used for the text in Inkscape, but the package 'color.sty' is not loaded}%
    \renewcommand\color[2][]{}%
  }%
  \providecommand\transparent[1]{%
    \errmessage{(Inkscape) Transparency is used (non-zero) for the text in Inkscape, but the package 'transparent.sty' is not loaded}%
    \renewcommand\transparent[1]{}%
  }%
  \providecommand\rotatebox[2]{#2}%
  \newcommand*\fsize{\dimexpr\f@size pt\relax}%
  \newcommand*\lineheight[1]{\fontsize{\fsize}{#1\fsize}\selectfont}%
  \ifx\svgwidth\undefined%
    \setlength{\unitlength}{193.53727566bp}%
    \ifx\svgscale\undefined%
      \relax%
    \else%
      \setlength{\unitlength}{\unitlength * \real{\svgscale}}%
    \fi%
  \else%
    \setlength{\unitlength}{\svgwidth}%
  \fi%
  \global\let\svgwidth\undefined%
  \global\let\svgscale\undefined%
  \makeatother%
  \begin{picture}(1,0.75345619)%
    \lineheight{1}%
    \setlength\tabcolsep{0pt}%
    \put(0,0){\includegraphics[width=\unitlength,page=1]{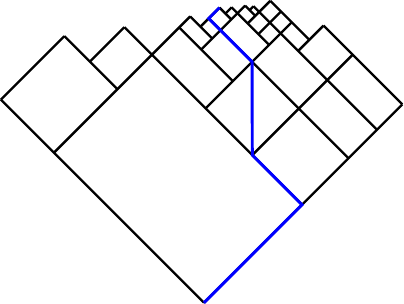}}%
    \put(0.68893406,0.09526113){\color[rgb]{0,0,1}\makebox(0,0)[lt]{\lineheight{1.25}\smash{\begin{tabular}[t]{l}$\wt{c}$\end{tabular}}}}%
  \end{picture}%
\endgroup%

    \caption{An example of a dynamic plane.}
    \label{fig:dynplane}
\end{figure}

Let $g=[c] \in \pi_1(M)$ be the automorphism of $\widetilde{M}$ specified by $c$. The dynamic plane $\widetilde{D}(c)$ is invariant under the action of $g$, and the quotient $D(c) = \widetilde{D}(c)/\langle g \rangle$ is homeomorphic to either: 
\begin{enumerate}
\item \emph{an open annulus} if $c$ is not a multiple of a branch loop and is orientation-preserving,
\item \emph{an open M\"obius band} if $c$ is not a multiple of a branch loop and is orientation-reversing, or
\item \emph{a half-open annulus} if $c$ is a multiple of a branch loop.
\end{enumerate}
We call $D(c)$ the \emph{dynamic annulus or M\"obius band} associated to $c$ accordingly. Note that as an abstract surface with a cell decomposition, $D(c)$ only depends on $c$, and not on the choice of the lift $\widetilde{c}$. This justifies our terminology.

Also, note that there is a map $D(c) \to B$, obtained by factoring the restriction of the covering map $\widetilde{B} \to B$ to $\widetilde{D}(c)$ through $D(c)$:
$$\begin{tikzcd}
& & D(c) \arrow[d] \\
\widetilde{D}(c) \arrow[r, hook] \arrow[rru] & \widetilde{B} \arrow[r] & B             
\end{tikzcd}$$

We denote the image of $\widetilde{G} \cap \widetilde{D}(c)$ in $D(c)$ as $G(c)$, and the image of $\widetilde{G_+} \cap \widetilde{D}(c)$ in $D(c)$ as $G_+(c)$. 
The map $D(c) \to B$ restricts to a simplicial map $G(c) \to G$ and $G_+(c) \to G_+$.

\begin{lemma} \label{lemma:invariancedynamic}
Suppose that $c_1$ and $c_2$ are two loops in the augmented dual graph $G_+$. If $c_1$ and $c_2$ are sweep-equivalent, then $D(c_1)=D(c_2)$.  
\end{lemma} 
\begin{proof}
It suffices to show that if $c_1$ is a strum of $c_2$ then $D(c_1) = D(c_2)$:
We can choose lifts $\widetilde{c_1}$ and $\widetilde{c_2}$ that are related by strumming across a $\langle g \rangle$-orbit of sectors, where $g=[c_1]=[c_2] \in \pi_1(M)$. 
Each sector in this orbit lies in $\widetilde{D}(c_1)$, thus we have $\widetilde{c_2} \subset \widetilde{D}(c_1)$. 
This implies that $\widetilde{D}(c_2) \subset \widetilde{D}(c_1)$, since every positive path starting from a point of $\widetilde{c_2}$ must lie on $\widetilde{D}(c_1)$. 
Similarly, since each sector in the collection lies in $\widetilde{D}(c_2)$, we have $\widetilde{D}(c_1) \subset \widetilde{D}(c_2)$, thus $\widetilde{D}(c_1)=\widetilde{D}(c_2)$. 
Finally, we have $D(c_1) = \widetilde{D}(c_1)/\langle g \rangle = \widetilde{D}(c_2)/\langle g \rangle = D(c_2)$. 
\end{proof}

Suppose the veering branched surface $B$ corresponds to a blown-up pseudo-Anosov flow $\phi^\sharp$. Recall the bijection $\mathcal{F}$ between the sweep-equivalence classes of loops in $G_+$ and the orbits of $\phi^\sharp$ from \Cref{thm:pavbscorr}.
If $\gamma$ is a closed orbit of $\phi^\sharp$, we define $D(\gamma)$ to be $D(c)$ for any loop $c$ in $G_+$ representing $\gamma$, i.e. any loop $c$ such that $\mathcal{F}(c) = \gamma$.
\Cref{lemma:invariancedynamic} ensures that this is well-defined.

In this case, we also write $G(c)$ as $G(\gamma)$ and $G_+(c)$ as $G_+(\gamma)$. 
We say that a loop $c$ of $G_+(\gamma)$ represents $\gamma$ if its image in $G_+$ represents $\gamma$.

\begin{prop} \label{prop:dynannuluscoreloops}
Let $\gamma$ be a closed orbit of $\phi^\sharp$. The map $D(\gamma) \to B$ determines a bijection between the set of loops in $G_+(\gamma)$ that represent $\gamma$ and the set of loops in $G_+$ that represent $\gamma$. Moreover, every loop in $G_+(\gamma)$ that represents $\gamma$ is embedded.
\end{prop}
\begin{proof}
By definition, every loop $\widehat{c}$ in $G_+(\gamma)$ that represents $\gamma$ maps to a loop $c$ in $G_+$ that represents $\gamma$. 

Conversely, given a loop $c$ in $G_+$ that represents $\gamma$, the projection of $\widetilde{c} \subset \widetilde{D}(c)$ to $D(c) = D(\gamma)$ is an embedded loop $\widehat{c}$ in $G_+(\gamma)$. Since $\widehat{c}$ maps to $c$ in $G_+$, it represents $\gamma$.

It is straightforward to verify that these two maps are inverses of each other.
\end{proof}

\subsection{Cores}

In what follows we assume that $B$ is the veering branched surface corresponding to the blow-up $(Y^\#, \phi^\#)$ of a pseudo-Anosov flow $(Y, \phi)$ along a collection $\mathcal{C}$ with respect to which $\phi$ has no perfect fits.

\begin{defn}[Cores]
A \emph{core} of the dynamic annulus/Möbius band $D(\gamma)$ is a compact subset $C$ such that the closure of each component of $D(\gamma) \backslash C$ is a half-open annulus given by a union of sectors, and whose boundary component is a loop in $G(\gamma)$ that represents $\gamma$ if $\gamma$ is orientation-preserving, or  $\gamma^2$ if $\gamma$ is orientation-reversing.
\end{defn}

Note that any loop in $G(\gamma)$ that represents $\gamma$ is a core for $D(\gamma)$. 
Conversely, we can always grow a core to contain a given loop in $G(\gamma)$ representing the orbit $\gamma$.

\begin{lemma} \label{lemma:coregrow}
Let $\gamma$ be a closed orbit, and $D(\gamma)$ its associated dynamic annulus/Möbius band. Given a core $C$ of $D(\gamma)$, and a loop $c$ in $G_+(\gamma)$ that represents 
\begin{itemize}
    \item $\gamma$ if $\gamma$ is orientation-preserving, or
    \item $\gamma^2$ if $\gamma$ is orientation-reversing,
\end{itemize} 
there exists a sequence of cores $C = C_0 \subset C_1 \subset \dots \subset C_n$ such that $C_i$ contains exactly one more sector than $C_{i-1}$ for each $i$ and $C_n$ contains $c$.
\end{lemma}
\begin{proof}
We perform the following operation on the component $A$ of $D(\gamma) \setminus C$ that meets $c$: 
Let $c'$ be the boundary component of the closure of $A$. 
Since the images of $c$ and $c'$ in $G_+$ both represent the same closed orbit, they are related by sweeping across finitely many sectors in $B$.
Lifting these sectors to $\widetilde{D}(c') \subset \widetilde{B}$ and projecting them to $D(c') = D(\gamma)$, we see that $c$ and $c'$ are related by sweeping across finitely many sectors in $D(\gamma)$ as well.
We can enlarge $C$ by adding these sectors one-by-one.

At each stage, one sector is removed from $A$, and the boundary component of $A$ is a loop of $G(\gamma)$ that represents $\gamma$ or $\gamma^2$, hence $C$ stays a core all the way. 
At the end of the construction, $C$ contains $c$.
\end{proof}

\begin{cor} \label{lemma:corecombine}
Given two cores $C$ and $C'$, there exists a sequence of cores $C = C_0 \subset C_1 \subset \dots \subset C_n$ such that $C_i$ contains exactly one more sector than $C_{i-1}$ for each $i$ and $C_n$ contains $C'$.    
\end{cor}
\begin{proof}
We apply \Cref{lemma:coregrow} to grow $C$ until it contains the boundary components of $D(\gamma) \setminus C'$.
\end{proof}

Our next goal is to show that this growing operation leads to a unique maximal core, see \Cref{prop:maxcore} below. We shall need two lemmas.

\begin{lemma} \label{lemma:divergingbranchrays}
Let $S$ and $S'$ be two sectors on a dynamic annulus/Möbius band $D(\gamma)$ where
\begin{itemize}
    \item the bottom vertex of $S'$ is a side vertex of $S$, and
    \item $S'$ is a toggle sector.
\end{itemize}
Let $r$ be the branch ray starting from the top side of $S$ that does not meet the bottom vertex of $S'$, and let $r'$ be the branch ray starting from top side of $S'$ that does not meet $S$, see \Cref{fig:divergingbranchrays}. 
Then $r$ and $r'$ never meet in $D(\gamma)$.
\end{lemma}

\begin{figure}
    \centering
    \selectfont\fontsize{6pt}{6pt}
    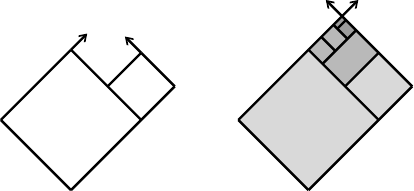
    \caption{The setup in \Cref{lemma:divergingbranchrays}.}
    \label{fig:divergingbranchrays}
\end{figure}

\begin{proof}
Suppose otherwise, then $r$, $r'$, and the bottom sides of $S$ and $S'$ bound a compact rectangle $R_0$.
Let $e$ and $e'$ be the second-to-top and top edge in the top side of $S$ that meet $S'$, and let $S_1$ and $S'_1$ be the sectors that have $e$ and $e'$ as bottom edges respectively.
Then the bottom vertex of $S'_1$ is a side vertex of $S_1$ and, by \Cref{prop:vbstogglefan}, $S'_1$ is a toggle sector.
Let $r_1$ be the branch ray starting from the top side of $S_1$ that does not meet the bottom vertex of $S'_1$, and let $r'_1$ be the branch ray starting from top side of $S'_1$ that does not meet $s_1$. Then $r_1$ is a subray of $r'_0$ and $r'_1$ is a subray of $r_0$, thus they must meet at a corner of $R_0$. 
In particular, $r_1$, $r'_1$, and the bottom sides of $S_1$ and $S'_1$ bound a proper sub-rectangle $R_1$. 

Repeating the argument, we get an infinite chain of rectangles $R_0 \supset R_1 \supset \dots$, which implies that $R_0$ contains infinitely many sectors. Contradiction.
\end{proof}

Using \Cref{lemma:divergingbranchrays}, we can show that there is an upper limit to the size of the cores.

\begin{lemma} \label{lemma:coresbounded}
There is a compact subset $\overline{C}$ such that every core is contained in $\overline{C}$.
\end{lemma}
\begin{proof}
Let $c_0$ be a loop in $G(\gamma)$ that represents $\gamma$. For each component $A$ of $D(\gamma) \backslash c_0$, we build a compact subset in $A$ as follows:
Let $(r_i)_{i \in \mathbb{Z}/k}$ be the set of branch rays in $A$ that start at a point of $\gamma_0$. For each $i$, let $S_{i,1},S_{i,2},\dots$ be the set of sectors whose bottom edges lie along $r_i$.
For each $j$, let $r'_{i,j}$ be the branch ray starting from the top side of $S_{i,j}$ that does not meet $S_{i,j-1}$. 

We claim that there is a value $m_i$ such that $r'_{i,j}$ meets $r_{i+1}$ if and only if $j \leq m_i$. 
Indeed, if $r'_{i,j}$ meets $r_{i+1}$ then $r'_{i,j-1}$ meets $r_{i+1}$, hence it suffices to show that $r'_{i,j}$ does not meet $r_{i+1}$ for some $j$.
Such a value of $j$ can be obtained by letting $j$ be the second smallest value such that $S_{i,j}$ is a toggle sector and applying \Cref{lemma:divergingbranchrays} (with $S_{i,j-1}$ playing the role of $S$ and $S_{i,j}$ playing the role of $S'$). 

We now take the union $\bigcup_{i \in \mathbb{Z}/k} \bigcup_{j=1}^{m_i} S_{i,j}$ as our compact subset of $A$. See \Cref{fig:coresbounded}. 
Let $\overline{C}$ be the union of these compact subsets as $A$ ranges over components of $D(\gamma) \backslash c_0$.
We claim that if $c$ is a loop in $G(\gamma)$ that represents $\begin{cases} \gamma & \text{if $\gamma$ is orientation-preserving} \\ \gamma^2 & \text{if $\gamma$ is orientation-reversing} \end{cases}$, then $c$ must be contained in $\overline{C}$.

\begin{figure}
    \centering
    \selectfont\fontsize{6pt}{6pt}
    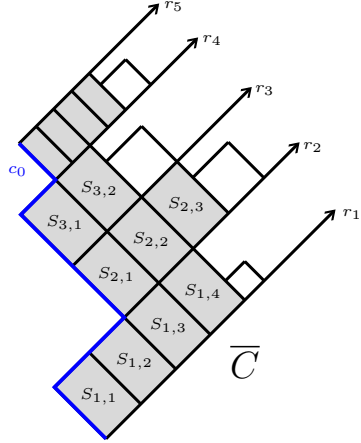
    \caption{Defining a compact region $\overline{C}$ so that every core is contained in $\overline{C}$.}
    \label{fig:coresbounded}
\end{figure}

Suppose otherwise, then since $c$ represents the same closed orbit as $c_0$ or $c_0^2$, we know that $c$ is related to $c_0$ or $c_0^2$ by sweeping across finitely many sectors. Hence, in each component of $D(\gamma) \backslash c_0$, in the above notation, $c$ meets each $r_i$ in a (possibly degenerate) segment. If $c$ is not contained in $\overline{C}$, then for at least one component of $D(\gamma) \backslash c_0$, $c$ meets $r'_{i,j}$ for some $i$ and some $j > m_i$. This contradicts the fact that $r'_{i,j}$ does not meet $r_{i+1}$.

To complete the proof, let $C$ be a core. By the claim above, the boundary of each component of $D(\gamma) \backslash C$ is contained in $\overline{C}$. Thus $C$ is contained in $\overline{C}$.
\end{proof}

\begin{thm} \label{prop:maxcore}
There exists a unique maximal core $C$ of $D(\gamma)$. 
The map $D(\gamma) \to B$ determines a bijection between the loops in $G_+(\gamma) \cap C$ that represent $\gamma$ and the loops in $G_+$ that represent $\gamma$.
Furthermore, given any core $C'$, there exists a sequence of cores $C' = C'_0 \subset C'_1 \subset \dots \subset C'_n$ such that $C'_i$ contains exactly one more sector than $C'_{i-1}$ for each $i$ and $C'_n=C$.
\end{thm} 
\begin{proof}
To show that there is a maximal core of $D(\gamma)$, we run the following operation: We start with some core $C$ of $D(\gamma)$. If there is a loop in $G_+(\gamma)$ that represents $\gamma$ but is not contained in $C_0$, we apply \Cref{lemma:coregrow} to enlarge $C$. This operation must eventually terminate by \Cref{lemma:coresbounded}, at which point we have a maximal core.

To show uniqueness of the maximal core, suppose $C_1$ and $C_2$ are both maximal cores. If $C_1 \neq C_2$, then without loss of generality, a boundary component of $D(\gamma) \backslash C_1$ is not contained in $C_2$. But we can then apply \Cref{lemma:coregrow} to enlarge $C_2$, contradicting maximality. 

By \Cref{prop:dynannuluscoreloops}, the map $D(\gamma) \to B$ determines a bijection between the loops in $G_+(\gamma)$ that represent $\gamma$ and the loops in $G_+$ that represent $\gamma$. But the maximal core $C$ contains every such loop in $G_+(\gamma)$, thus the set of loops in $G_+(\gamma)$ that represent $\gamma$ and the set of loops in $G_+(\gamma) \cap C$ that represent $\gamma$ coincide.

Finally, the last statement is \Cref{lemma:corecombine}.
\end{proof}

\section{Computation of the differential for sleek summands} \label{sec:sleeksummands}

Let $(Y,\phi)$ be a pseudo-Anosov flow and $\mathcal{C}$ be a collection of closed orbits of $\phi$ such that $\phi$ has no perfect fits relative to $\mathcal{C}$. In this section, we explain how to compute the homology of any sleek summand of $SFC(\phi, \mathcal{C},\ts)$. Our main goal is to show the following.

\begin{thm} \label{thm:insulatedorbithomology}
In any sleek $\ts$-grading we have that $\dim H_*(SFC(\phi, \mathcal{C},\widetilde{\mathfrak{s}})) = 1$.
\end{thm}

We shall split the discussion in two parts. First, given a sleek $\ts$-grading we describe a combinatorial model  for the chain complex  $SFC(\phi, \mathcal{C},\ts)$ for which it is easy to prove that the homology is one dimensional. For this purpose we use the theory of comapct cores we exposed in the previous section. 
Then we identify $SFC(\phi, \mathcal{C},\ts)$ with this abstract model.

\subsection{Combinatorial chain complexes} \label{subsec:combinchaincomplex}

Let $\gamma$ be a closed orbit of the blown-up flow $(Y^\#, \phi^\#)$ associated to $(\phi,\mathcal{C})$.
In this subsection, we will associate a chain complex to each core of the dynamic annulus/Möbius band $D(\gamma)$. 
We will also show that the homology of each of these chain complexes has dimension $1$. 

Suppose that $C\subset D(\gamma)$ is a core for $D(\gamma)$. We define $CC(C)$ to be the $\mathbb{F}_2$-vector space with basis the loops in $G_+(\gamma) \cap C$ that represent the orbit $\gamma$. (Here $CC$ stands for Combinatorial Chain complex.)

Recall the notion of red and blue sectors from \Cref{defn:blueredsectors}. Given a basis element $c$ of $CC(C)$, we consider the sets
\begin{align*}
\mathcal{R}(c) &:=  \{c'\subset G_+(\gamma) \cap C \mid c' \text{ is a strum of } c \text{ across a single red sector}\} \\
\mathcal{L}(c) &:=  \{c'\subset G_+(\gamma) \cap C \mid c \text{ is a strum of } c' \text{ across a single blue sector}\}
\end{align*}
and we define 
\[\partial: CC(C)\to CC(C)\] using the canonical basis of $CC(C)$ via the formula
\begin{align*}
\partial c =& \sum_{c'\in \mathcal{R}(c)} c' + \sum_{c'\in \mathcal{L}(c)} c'  \ . 
\end{align*}

\begin{lemma} 
We have that $\partial^2=0$, that is, $(CC(C), \partial)$ is a chain complex.
\end{lemma}
\begin{proof}
We argue that the terms in $\partial^2 c$ occur in pairs. Indeed, suppose $c''$ is a term in $\partial^2 c$. Then $c''$ is a term in $\partial c'_1$ where $c'_1$ is a term in $\partial c$, thus either
\begin{enumerate}
    \item $c'_1$ is obtained from $c$ by strumming across a red sector $S$ and $c''$ is obtained from $c'_1$ by strumming across another red sector $S'$, or
    \item $c$ is obtained from $c'_1$ by strumming across a blue sector $S$ and $c'_1$ is obtained from $c''$ by strumming across another blue sector $S'$, or
    \item $c$ is obtained from $c'_1$ by strumming across a blue sector $S$ and $c''$ is obtained from $c'_1$ by strumming across a red sector $S'$, or    
    \item $c'_1$ is obtained from $c$ by strumming across a red sector $S$ and obtained from $c''$ by strumming across a blue sector $S'$.
\end{enumerate}

In cases (1) and (3), we define $c'_2$ by strumming $c$ across $S'$. Then $c'_2$ is a term in $\partial c$ and $c''$ is a term in $\partial c'_2$, thus $c''$ appears twice in $\partial^2 c$. A similar construction works in cases (2) and (4).
\end{proof}

\begin{lemma} \label{lemma:cchomologyinduct}
Suppose $C \subset C'$ are cores where $C'$ contains exactly one more sector $S$ than $C$. Then the homology of $CC(C)$ and $CC(C')$ are isomorphic.
\end{lemma}
\begin{proof}
We first suppose that $S$ is red.
Let $e_1$ be the path in $G_+(\gamma) \cap C'$ which lies on the boundary of $S$ and consists of the top and bottom sides of $S$ that lie in the interior of $D(\gamma) \backslash C$. Let $e_2$ be the edge of $G_+(\gamma) \cap C'$ that connects the bottom vertex of $S$ to the top vertex of $S$.

Since the basis of $CC(C)$ is a subset of the basis of $CC(C')$, the chain complex $CC(C)$ is naturally a subspace of $CC(C')$. 
In fact, we claim that $CC(C)$ is a subcomplex of $CC(C')$. Suppose otherwise, then there is a loop $c$ contained in $C$ and a loop $c'$ contained in $C'$ but not in $C$ such that $c'$ appears as a term in $\partial c$. 
By definition of $\partial c$, either $c'$ is a strum of $c$ across a single red sector $S'$, or $c$ is a strum of $c'$ across a single blue sector $S'$.
In the former case, if $S' \neq S$, then $S'$ is contained in $C$, and $c'$ is contained in $C$, contradicting our assumption. Thus $S' = S$. But then since $c$ is contained in $C$, we cannot strum $c$ across $S'=S$, hence we reach a contradiction as well.
In the latter case, since $S'$ is blue and $S$ is red, we have $S' \neq S$, which gives us the contradiction that $c'$ is contained in $C$ as above.

Thus we have a two-step filtration $CC(C)\subset CC(C')$. The quotient complex $\overline{CC}(C') = CC(C')/CC(C)$ can be regarded as the vector space generated by the loops in $G_+(\gamma) \cap C'$ that pass through $e_1$ or $e_2$. 
Consequently, we can decompose the quotient complex $\overline{CC}(C')$ into $\overline{CC}(C')_1 \oplus \overline{CC}(C')_2$ where $\overline{CC}(C')_i$ is the subspace generated by the loops that pass through $e_i$.

We claim that $\overline{CC}(C')_1$ is a subcomplex of $\overline{CC}(C')$. Suppose otherwise, then there is a loop $c_1$ passing through $e_1$ and a loop $c_2$ passing through $e_2$ such that $c_2$ appears as a term in $\partial c_1$. 
By definition of $\partial c_1$, either $c_2$ is a strum of $c_1$ across a single red sector $S'$, or $c_1$ is a strum of $c_2$ across a single blue sector $S'$.
But since $e_1$ is a strum of $e_2$ across $S$, the latter must be true, with $S'=S$. We reach a contradiction since $S'$ is blue while $S$ is red.

Because of the claim, we can write the differential of $\overline{CC}(C')$ as $\begin{bmatrix} \partial_1 & \alpha \\ 0 & \partial_2 \end{bmatrix}$, where $\partial_i$ is the restriction of the differential to $\overline{CC}(C')_i$, for $i=1,2$, and $\alpha: \overline{CC}(C')_2 \to \overline{CC}(C')_1$ is a chain map.
In other words, the quotient complex $\overline{CC}(C')$ can be identified with the mapping cone of $\alpha: \overline{CC}(C')_2 \to \overline{CC}(C')_1$. Thus we have an exact triangle
\[
\begin{tikzcd}
 H_*(\overline{CC}(C')_2) \arrow[rr, "\alpha_* "] & & H_*(\overline{CC}(C')_1) \arrow[dl] \\ & H_*(\overline{CC}(C')) \arrow[ul]
\end{tikzcd}
\] 
In fact, we can describe $\alpha$ explicitly as the map that takes $c_2$ to $c_1 = (c_2 \backslash e_2) \cup e_1$. 
In particular, we observe that $\alpha$ is an isomorphism, thus in the diagram above, the horizontal map is an isomorphism, showing that $H_*(\overline{CC}(C'))=0$.

Now looking at the exact triangle
\[
\begin{tikzcd}
 H_*(CC(C)) \arrow[rr, "\iota_*"] & & H_*(CC(C')) \arrow[dl, "\pi_*"] \\ & H_*(\overline{CC}(C')) \arrow[ul, "\delta"]
\end{tikzcd}
\]
associated to the short exact sequence 
\[
\begin{tikzcd}
 0 \arrow[r, ] &CC(C) \arrow[r, "\iota"] & CC(C') \arrow[r, "\pi"] &  \overline{CC}(C') \arrow[r] & 0 
\end{tikzcd}
\] 
via the Snake Lemma, we deduce that $CC(C)$ and $CC(C')$ have the same homology.

A similar argument holds when $s$ is blue, where we consider $CC(C)$ as a quotient complex of $CC(C')$ instead.
\end{proof}

\begin{prop} \label{prop:cchomology}
For every core $C$, $CC(C)$ is a chain complex with homology $\mathbb{F}$.
\end{prop}
\begin{proof}
By \Cref{prop:maxcore} and \Cref{lemma:cchomologyinduct}, it suffices to compute $CC(C)$ for one core $C$. We choose $C$ to be a loop in $G_+(\gamma)$ that represents $\gamma$, then $CC(C) \cong \mathbb{F}$ with a zero differential, hence its homology is $\mathbb{F}$.
\end{proof}

\subsection{Identification of chain complexes}

Let $B \subset Y^\#$ be the veering branched surface associated to  the blow-up  $(Y^\#, \phi^\#)$  along $\mathcal{C}$.
In this subsection, we will identify the summand $SFC(\phi, \mathcal{C},\ts)$ with the chain complex of the maximal core of the corresponding sleek orbit. Together with \Cref{prop:cchomology}, this implies \Cref{thm:insulatedorbithomology}.

\begin{prop} \label{prop:sfc=cc}
Suppose $SFC(\phi, \mathcal{C},\ts)$ is a sleek summand, and that $\gamma=\gamma_{\ts}$ is the sleek orbit associated to $\ts$.
Let $C$ be the maximal core of $D(\gamma)$. Then we have an isomorphism of chain complexes $SFC(\phi, \mathcal{C},\ts) \cong CC(C)$.
\end{prop}
\begin{proof}
By \Cref{prop:maxcore}, the map $D(\gamma) \to B$ determines a natural bijection between the generators of $SFC(\phi, \mathcal{C},\ts)$ and the generators of $CC(C)$.
Throughout this proof, we will apply this identification implicitly. Our task is to identify the differentials of $SFC(\phi, \mathcal{C},\ts)$ and $CC(C)$. 
In turn, we have to determine the effective domains between the generators that counts towards the differential.

By \Cref{prop:transhomclassobsdomain}, two generators that are connected by an effective domain have a common strum. Thus we fix a loop $c_0$ in $G$ and consider the set $\mathcal{S}$ of loops in $G_+$ that have $c_0$ as a strum. 
We refer to a point of $c_0$ where $c_0$ locally changes from lying on one branch loop to another branch loop as a \emph{turn}. In other words, a turn is a non-smooth point of $c_0$.
Let $\mathcal{T}$ be the set of turns of $c_0$. Observe that $\mathcal{S}$ can be identified with a subset of the power set of $\mathcal{T}$. Indeed, a loop in $G$ can be recovered by specifying which turns arise from strumming to $c_0$. Thus from now on we identify the elements of $\mathcal{S}$ with subsets of $\mathcal{T}$. 

For each $t \in \mathcal{T}$, we define $D(t)$ to be the elementary domain whose bottom left or bottom right corner is the component of $\xb$ on the $\alpha$-curve corresponding to $t$. (Recall \Cref{prop:globalcombinatorics} for the combinatorics of the Heegaard diagram.)
Equivalently, note that the elementary domains are in correspondence with the top sides of sectors, where the $\alpha$-curves met by an elementary domain $D$ corresponds to the triple points lying in the top side corresponding to $D$. Under this correspondence, the elementary domain $D(t)$ corresponds to the top side contained in $c_0$ starting at $t$.
Observe that since $\ts$ is sleek, $c_0$ is embedded, thus the elementary domains $D(t)$ have mutually disjoint interiors.

Now suppose there is an effective domain $D$ connecting $c_1 \in \mathcal{S}$ to $c_2 \in \mathcal{S}$. We consider the sets
\begin{align*}
\mathcal{T}^L_1 &:= \{t \in \mathcal{T} \mid \text{$t$ is a side vertex of a blue sector, $t \not\in c_1$, and $t \in c_2$}\} \\
\mathcal{T}^L_2 &:= \{t \in \mathcal{T} \mid \text{$t$ is a side vertex of a blue sector, $t \in c_1$, and $t \not\in c_2$}\} \\
\mathcal{T}^R_1 &:= \{t \in \mathcal{T} \mid \text{$t$ is a side vertex of a red sector, $t \not\in c_1$, and $t \in c_2$}\} \\
\mathcal{T}^R_2 &:= \{t \in \mathcal{T} \mid \text{$t$ is a side vertex of a red sector, $t \in c_1$, and $t \not\in c_2$}\}.
\end{align*}
We claim that $\mathcal{T}^L_2 = \mathcal{T}^R_1 = \varnothing$ and $D = \sum_{t \in \mathcal{T}^L_1} D(t) + \sum_{t \in \mathcal{T}^R_2} D(t)$. 

Take $t \in \mathcal{T}^L_1$. 
Then the Heegaard diagram near $D(t)$ is of the form in \Cref{fig:sfc=cc}.
Recall from \Cref{prop:admdomaincorners} that $D$ is embedded and the boundary of the support of $D$ consists of alternating arcs lying on $\alpha$ and $\beta$. 
Moreover, since $D$ connects $c_1$ to $c_2$, each $\alpha$ arc connects a component of $c_1$ to $c_2$ while each $\beta$-arc connects a component of $c_2$ to $c_1$. 
Thus $D$ must contain $D(t)$.
Similarly, $D$ must contain $D(t)$ for every $t \in \mathcal{T}^R_2$.

On the other hand, suppose we have $t \in \mathcal{T}^L_2$, then the Heegaard diagram near $D(t)$ is of the form in \Cref{fig:sfc=cc} but with $c_1$ and $c_2$ interchanged. In this case, we have a contradiction to the hypothesis that $D$ connects $c_1$ to $c_2$ (as opposed to connecting $c_2$ to $c_1$). Similarly, we have a contradiction if an element $t \in \mathcal{T}^R_1$ exists. 
This shows our claim that $\mathcal{T}^L_2 = \mathcal{T}^R_1 = \varnothing$ and $D = \sum_{t \in \mathcal{T}^L_1} D(t) + \sum_{t \in \mathcal{T}^R_2} D(t)$.

\begin{figure}
    \centering
    \selectfont\fontsize{10pt}{10pt}
\begingroup%
  \makeatletter%
  \providecommand\color[2][]{%
    \errmessage{(Inkscape) Color is used for the text in Inkscape, but the package 'color.sty' is not loaded}%
    \renewcommand\color[2][]{}%
  }%
  \providecommand\transparent[1]{%
    \errmessage{(Inkscape) Transparency is used (non-zero) for the text in Inkscape, but the package 'transparent.sty' is not loaded}%
    \renewcommand\transparent[1]{}%
  }%
  \providecommand\rotatebox[2]{#2}%
  \newcommand*\fsize{\dimexpr\f@size pt\relax}%
  \newcommand*\lineheight[1]{\fontsize{\fsize}{#1\fsize}\selectfont}%
  \ifx\svgwidth\undefined%
    \setlength{\unitlength}{201.3182993bp}%
    \ifx\svgscale\undefined%
      \relax%
    \else%
      \setlength{\unitlength}{\unitlength * \real{\svgscale}}%
    \fi%
  \else%
    \setlength{\unitlength}{\svgwidth}%
  \fi%
  \global\let\svgwidth\undefined%
  \global\let\svgscale\undefined%
  \makeatother%
  \begin{picture}(1,0.61109717)%
    \lineheight{1}%
    \setlength\tabcolsep{0pt}%
    \put(0,0){\includegraphics[width=\unitlength,page=1]{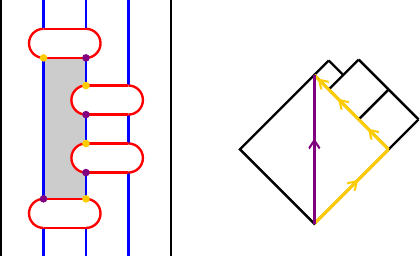}}%
    \put(0.13163811,0.29080285){\color[rgb]{0,0,0}\makebox(0,0)[lt]{\lineheight{1.25}\smash{\begin{tabular}[t]{l}$D$\end{tabular}}}}%
    \put(0.67067933,0.25651916){\color[rgb]{0.50196078,0,0.50196078}\makebox(0,0)[lt]{\lineheight{1.25}\smash{\begin{tabular}[t]{l}$c_1$\end{tabular}}}}%
    \put(0.82785422,0.38114082){\color[rgb]{1,0.8,0}\makebox(0,0)[lt]{\lineheight{1.25}\smash{\begin{tabular}[t]{l}$c_2$\end{tabular}}}}%
  \end{picture}%
\endgroup%

    \caption{If there is an effective domain $D$ connecting $c_1 \in \mathcal{S}$ to $c_2 \in \mathcal{S}$ with Lipshitz index $1$, then $c_1$ and $c_2$ differ by a turn $t$ and $D=D(t)$.} 
    \label{fig:sfc=cc}
\end{figure}
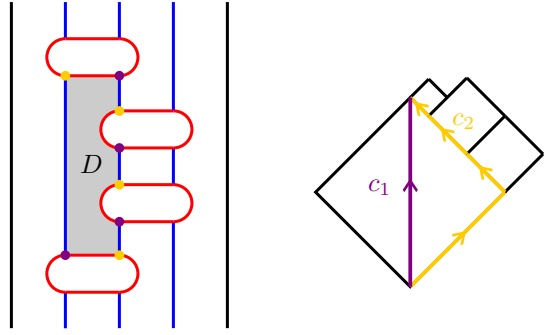

Since each elementary domain $D(t)$ is an empty polygon in the sense of \Cref{defn:emptypolygon}, and $D(t)$ have mutually disjoint interiors, we can apply \Cref{defn:lipshitzindex} to compute the Lipshitz index of $D$ to be $\mu(D) = |\mathcal{T}^L_1| + |\mathcal{T}^R_2|$. Thus $D$ contributes to the differential of $SFC(\phi, \mathcal{C},\ts)$ only if $\mathcal{T}^L_1 = \varnothing$ and $\mathcal{T}^R_2$ consists of exactly one element, or vice versa.
Furthermore, in this case, by \Cref{thm:polygonscountonce}, $D$ contributes a term $c_2$ for $c_2 \in \mathcal{R}(c_1)$ if $\mathcal{T}^L_1 = \varnothing$, or contributes a term $c_2$ for $c_2 \in \mathcal{L}(c_1)$ if $\mathcal{T}^R_2 = \varnothing$.
Thus the differentials of $SFC(\phi,\mathcal{C},\ts)$ and $CC(C)$ coincide.
\end{proof}

\section{Applications}

\subsection{Orientable veering branched surfaces}

We say that a veering branched surface $B$ is \emph{orientable} if one can consistently assign an orientation to the tangent plane $T_pB$ as $p$ varies over points of $B$.

If a veering branched surface $B$ is orientable, then we say that a branch loop $b$ is \emph{eastward} if at every point of $b$, (orientation of $b$, maw coorientation of $b$) is a positively oriented basis of $TB$. Otherwise we say that $b$ is \emph{westward}. See \Cref{fig:orientablevbslocal}.

\begin{figure}
    \centering
    \selectfont\fontsize{8pt}{8pt}
\begingroup%
  \makeatletter%
  \providecommand\color[2][]{%
    \errmessage{(Inkscape) Color is used for the text in Inkscape, but the package 'color.sty' is not loaded}%
    \renewcommand\color[2][]{}%
  }%
  \providecommand\transparent[1]{%
    \errmessage{(Inkscape) Transparency is used (non-zero) for the text in Inkscape, but the package 'transparent.sty' is not loaded}%
    \renewcommand\transparent[1]{}%
  }%
  \providecommand\rotatebox[2]{#2}%
  \newcommand*\fsize{\dimexpr\f@size pt\relax}%
  \newcommand*\lineheight[1]{\fontsize{\fsize}{#1\fsize}\selectfont}%
  \ifx\svgwidth\undefined%
    \setlength{\unitlength}{111.62525688bp}%
    \ifx\svgscale\undefined%
      \relax%
    \else%
      \setlength{\unitlength}{\unitlength * \real{\svgscale}}%
    \fi%
  \else%
    \setlength{\unitlength}{\svgwidth}%
  \fi%
  \global\let\svgwidth\undefined%
  \global\let\svgscale\undefined%
  \makeatother%
  \begin{picture}(1,0.68153875)%
    \lineheight{1}%
    \setlength\tabcolsep{0pt}%
    \put(0,0){\includegraphics[width=\unitlength,page=1]{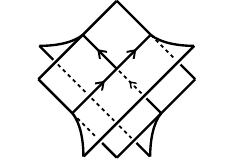}}%
    \put(0.67630173,0.56440797){\color[rgb]{0,0,0}\makebox(0,0)[lt]{\lineheight{1.25}\smash{\begin{tabular}[t]{l}Eastward\end{tabular}}}}%
    \put(-0.00321947,0.56440797){\color[rgb]{0,0,0}\makebox(0,0)[lt]{\lineheight{1.25}\smash{\begin{tabular}[t]{l}Westward\end{tabular}}}}%
  \end{picture}%
\endgroup%

    \caption{When a veering branched surface $TB$ is orientable, we can label its branch loops as westward or eastward. In this figure, the orientation on $TB$ is induced from the standard one on the page.}
    \label{fig:orientablevbslocal}
\end{figure}

Every veering branched surface admits a 2-fold orientable cover. More precisely, if $B$ is a veering branched surface on a compact 3-manifold $M$, then either $B$ is orientable, or there exists a 2-fold cover $\widehat{M}$ of $M$ such that the lift $\widehat{B}$ of $B$ is orientable. Namely, the 2-fold cover is determined by the homomorphism $\pi_1 M \cong \pi_1 B \to \mathbb{Z}/2$ sending a loop $\alpha$ to $0$ if the pullback of $TB$ along $\alpha$ is orientable and $1$ otherwise.

\begin{lemma} \label{lemma:orientablevbsbranchloopemb}
Let $B$ be an orientable veering branched surface.
Any collection of eastward branch loops is an embedded multi-loop. 
Similarly, any collection of westward branch loops is an embedded multi-loop.
\end{lemma}
\begin{proof}
Each triple point meets one eastward and one westward branch loop, so eastward branch loops cannot intersect each other (nor self-intersect), and same for westward branch loops.
\end{proof}

\begin{cor}
Let $B$ be an orientable veering branched surface. Suppose $B$ has $2N$ branch loops. Then $\dim SFH(Q) \geq 2^{N+1}$.
\end{cor}
\begin{proof}
There are $N$ eastward branch loops. From these we can form $2^N$ multi-loops that are unions of eastward branch loops, including the empty union. By \Cref{lemma:orientablevbsbranchloopemb}, these are all embedded, so by \Cref{prop:embbranchloopssleek}, they are all sleek.
Similarly, there are $2^N$ multi-loops that are unions of westward branch loops. 
Accounting for double-counting the empty union, we have $2^{N+1}-1$ such branch multi-loops, determining $2^{N+1}-1$ sleek summands. 
By \Cref{thm:insulatedorbithomology}, each sleek summand contributes $1$ to $\dim SFH(Q)$.

The last dimension comes from the state $\xt$. This state corresponds to the multi-loop consisting of all bottom-to-top loops. In particular, it is not sleek since we can strum any one of the edges and create a self-intersection, so we cannot apply \Cref{thm:insulatedorbithomology}. But this also implies that $\xt$ does not lie in any sleek $\widetilde{\ts}$-grading, and we know that $x^\top$ determines a nonvanishing class by \cite[Theorem 1.4]{AT25a}, so we have an additional dimension that is not counted above.
\end{proof}

\subsection{Pseudo-Anosov maps}

Let $S$ be a closed oriented surface. An orientation-preserving homeomorphism $f:S \to S$ is \emph{pseudo-Anosov} if there exists a pair of transverse measured singular foliation $(\ell^s,\ell^u)$ such that $f$ contracts the leaves of $\ell^s$ and dilates the leaves of $\ell^u$ by a common factor $\lambda > 1$.

If $f$ is pseudo-Anosov, then the suspension flow on the mapping torus of $f$ is a pseudo-Anosov flow $\phi_f$. 
There is a correspondence between the periodic orbits of $f$ and the closed orbits of $\phi_f$. 
More precisely, we define the \emph{period} of a closed orbit $\gamma$ of $\phi_f$ to be the number of times which it intersects the fiber surface $S$.
If $x$ is a periodic orbit of $f$ of period $p$, then the suspension of $x$ is a primitive closed orbit of period $p$. Conversely, a primitive closed orbit of period $p$ determines a periodic orbit of period $p$, namely, the $p$ points where it intersects the fiber surface.

Given a finite collection $\mathfrak{c}$ of periodic orbits of $f$, we can obtain a homeomorphism $f^\sharp$ on $S^\sharp = S \backslash \nu(\mathfrak{c})$ by blowing up $f$ along the points in $\mathfrak{c}$. 
The suspension flow of $f^\sharp$ is the blow-up of $\phi_f$ along the collection $\mathcal{C}$ of closed orbits corresponding to $\mathfrak{c}$.
The discussion in the previous paragraph on the correspondence between periodic points and closed orbits carries through for blown-up pseudo-Anosov maps.

Finally, we observe that the flow $\phi_f$ has no perfect fits relative to any nonempty finite collection $\mathcal{C}$ of closed orbits that contains the singular orbits. 
In this context, we say that the corresponding veering branched surface $B$ is \emph{layered}. 

\begin{prop} \label{prop:leastorbitpantsirreducible}
Let $f^\sharp$ be a blown-up pseudo-Anosov map and let $B$ be the corresponding layered veering branched surface. Suppose $P$ is the least period among the periodic points of $f^\sharp$.
Then every closed multi-orbit $\gamma$ of $\phi^\sharp$ of period $n \in [P,2P-1]$ is a closed orbit and is pants-irreducible.
\end{prop}
\begin{proof}
We first show that $\gamma$ can consist of only one orbit. Otherwise suppose $\gamma = \{\gamma_1,\dots,\gamma_k\}$ with $k \geq 2$. 
Then $[\gamma] = [\gamma_1] + \dots + [\gamma_k]$ in $H_1(Y^\sharp)$. In particular, the periods of $\gamma_i$ must sum up to $n \leq 2P-1$. But since the period of each $\gamma_i$ must be positive, the period of at least one $\gamma_i$ must be less than $P$. This gives periodic orbits of $f^\sharp$ of period less than $P$, contradicting the definition of $P$.

Now suppose $\gamma$ can be resolved along a immersed Fried pants $P$. Without loss of generality suppose $\gamma$ is the positive boundary of $P$. Let $\gamma'$ be the negative boundary of $P$. Then $\gamma'$ consists of two closed orbits, say $\gamma'_1$ and $\gamma'_2$.

Meanwhile, $\gamma$ is homologous to $\gamma'$, so $[\gamma] = [\gamma'_1] + [\gamma'_2]$ in $H_1(Y^\sharp)$. 
Reasoning as in the first paragraph, we reach a contradiction.
\end{proof}

\begin{proof}[Proof of \Cref{thm:introperiodicpoint}]
Recall from \Cref{lemma:difference} and \Cref{lemma:epsilon=mu=gamma} that under the identification $\Spinc(Y^\sharp) \cong H_1(Y^\sharp)$ that sends $\overline{\mathfrak{s}_{\phi^\sharp}}$ to $0$, the $\spinc$-structure $\mathfrak{s}(\x)$ associated to a state $\x$ is $[\gamma_\x] \in H_1(Y^\sharp)$.

For $n \in [P,2P-1]$, every closed multi-orbit of period $n$ is a closed orbit and is pants-irreducible by \Cref{prop:leastorbitpantsirreducible}, thus sleek by \Cref{prop:pantsirreducibleimplysleek}, hence determines a sleek $\ts$-grading in $SFC(\phi,\mathcal{C},n)$.
Conversely, suppose $\ts(\x)$ is a $\ts$-grading where $\langle \mathfrak{s}(\x), [S^\sharp] \rangle = n$, then $\gamma_\x$ is a closed multi-orbit of period $n$ thus is sleek by \Cref{prop:leastorbitpantsirreducible} and \Cref{prop:pantsirreducibleimplysleek} as above.
This argument implies that every $\ts$-grading in $SFC(\phi,\mathcal{C},n)$ is sleek, and there are as many $\ts$-gradings in $SFC(\phi,\mathcal{C},n)$ as there are period $n$ closed orbits, which is in turn equals to $n$ times the number of period $n$ periodic points of $f^\sharp$.

Thus we conclude the theorem by \Cref{thm:insulatedorbithomology}.
\end{proof}

\begin{proof}[Proof of \Cref{cor:introknotfixedpoint}]
Since $f^\sharp$ does not have interior singularities, it is the blow-up of a pseudo-Anosov map $f:S \to S$ at a fixed point $\mathfrak{c}$. By the Poincare-Hopf theorem, $\mathfrak{c}$ must be a $(4g-2)$-pronged singular point of $f$.
\Cref{thm:introperiodicpoint} implies that 
$$\dim SFH(Y^\sharp,6g-4) = \text{\# fixed points of $f^\sharp$} = \text{\# interior fixed points of $f^\sharp$} + (8g-4).$$

Meanwhile, since we are assuming that the degeneracy slope of $f^\sharp$ is the meridian of $K$, $Y^\sharp$ is the sutured manifold obtained by placing $8g-4$ sutures along the meridian of $Y \backslash \nu(K)$. 
Let $Y^\natural$ be the sutured manifold obtained by placing two sutures along the meridian of $Y \backslash \nu(K)$ instead. 
Then $Y^\sharp$ can be decompose along product annuli into $Y^\natural$ and $4g-3$ solid tori, $T_1,\dots,T_{4g-3}$, each with 4 longitudinal sutures.
By \cite[Proposition 8.6]{Juh08}, this implies that $SFH(Y^\sharp) \cong SFH(Y^\natural) \otimes SFH(T_1) \otimes \dots \otimes SFH(T_{4g-3})$.
By \cite[Example 7.5]{Juh10}, $SFH(T_i) \cong \mathbb{F}\langle \z^{\top} \rangle \oplus \mathbb{F}\langle \z^{\bot} \rangle$, for each $i=1,\dots,4g-3$. Here the difference in $\spinc$-gradings for the generators $\z^{\top}$, $\z^{\bot}$ is a meridian curve in $Y \backslash \nu(K)$. 
Thus the image of $\xb$ in $SFH(Y^\natural) \otimes SFH(T_1) \otimes \dots \otimes SFH(T_{4g-3})$ is of the form $(\xb)^\natural \times (\z^{\bot})^{\otimes (4g-3)}$, and if we identify $\Spinc(Y^\natural) \cong H_1(Y^\natural)$ so that $\mathfrak{s}((\xb)^\natural)$ is sent to $0$, then we have 
$$\dim SFH(Y^\sharp,6g-3) = \dim SFH(Y^\natural,2g)=1$$
and
\begin{align*}
\dim SFH(Y^\sharp,6g-4) &= \dim SFH(Y^\natural,2g-1)+(4g-3) \dim SFH(Y^\natural,2g) \\
&= \dim SFH(Y^\natural,2g-1)+(4g-3).
\end{align*}

Now, by \cite[Proposition 9.2]{Juh06} that $\widehat{HFK}(Y,K) \cong SFH(Y^\natural)$. More specifically, under the identification $\Spinc(Y^\natural) \cong H_1(Y^\natural)$ that we have chosen, we have 
$$\widehat{HFK}(Y,K,n) \cong SFH(Y^\natural,n+g)$$
for every $n$. 
Combining all the equalities, we have 
\begin{align*}
\dim \widehat{HFK}(Y,K,g-1) &= \dim SFH(Y^\natural,2g-1) \\
&= \dim SFH(Y^\sharp,6g-4)-(4g-3) \\
&= \text{\# interior fixed points of $f^\sharp$} + (4g-1)
\end{align*}
as claimed.
\end{proof}

\bibliographystyle{alpha}

\bibliography{bib.bib}

\end{document}